\documentclass[11pt,reqno]{amsart}
\usepackage{amsmath,amsfonts,amssymb,amsthm,amscd,comment,euscript}
\usepackage[all]{xy}
\usepackage{graphicx}
\usepackage{mathptmx}
\usepackage{enumerate} 
\usepackage[colorlinks=true]{hyperref} 
\usepackage[usenames,dvipsnames]{xcolor}
\usepackage{enumitem}

 \calclayout

\numberwithin{equation}{section}

\theoremstyle{plain}

\newtheorem{lemma}{Lemma}[section]
\newtheorem{corollary}[lemma]{Corollary}
\newtheorem{proposition}[lemma]{Proposition}
\newtheorem{theorem}[lemma]{Theorem}

\theoremstyle{definition}

\newtheorem{remark}[lemma]{Remark}
\newtheorem{definition}[lemma]{Definition}

\renewcommand{\phi}{\varphi}
\renewcommand{\leq}{\leqslant}
\renewcommand{\geq}{\geqslant}
\renewcommand{\epsilon}{\varepsilon}

\renewcommand{\kappa}{\varkappa}

\DeclareMathOperator{\Fr}{Fr}
\DeclareMathOperator{\Mtr}{\mathsf{Mtr}}
\DeclareMathOperator{\Mtrc}{\mathcal{M}\mathsf{tr}}
\DeclareMathOperator{\spec}{Spec}
 \DeclareMathOperator{\cyl}{cyl}
 \DeclareMathOperator{\diag}{diag}

\DeclareMathOperator{\sd}{sd}

\DeclareMathOperator{\Aff}{Aff} 
\DeclareMathOperator{\Hom}{Hom} 
\DeclareMathOperator{\corr}{Corr} 
\DeclareMathOperator{\End}{End} \DeclareMathOperator{\id}{id}

 \DeclareMathOperator{\Mor}{Mor}
\DeclareMathOperator{\Alg}{Alg} \DeclareMathOperator{\colim}{colim}
\DeclareMathOperator{\Ho}{Ho} \DeclareMathOperator{\Fun}{\sf Fun}
 
 \DeclareMathOperator{\kgl}{\mathsf{kgl}}
\DeclareMathOperator{\pt}{pt} 
 \DeclareMathOperator{\im}{Im}

\DeclareMathOperator{\smaff }{AffSm} \DeclareMathOperator{\Ar}{Ar}
 
 \DeclareMathOperator{\Mod}{Mod}
 \DeclareMathOperator{\Ob}{Ob}

\DeclareMathOperator{\nis}{\mathsf{nis}}

\newcommand{\lra}[1]{\bl{#1}\longrightarrow\relax}
\newcommand{\bl}[1]{\buildrel #1\over}

\newcommand{\cc}{\mathcal}
\newcommand{\bb}{\mathbb}

\newcommand{\op}{{\textrm{\rm op}}}

\newcommand{\aha}{\Alg^u_k}

\newcommand{\ahaw}{{\Alg_{k}}}

\newcommand{\uhom}{\underline{\Hom}}
\newcommand{\shnis}{SH^{\nis}_{S^{1}}}

\makeatother

\begin{document}

\footskip30pt

\title{K-theory and matrix transfers}
\author{Grigory Garkusha}

\address{Department of Mathematics, Swansea University, Fabian Way, Swansea SA1 8EN, UK}
\email{G.Garkusha@swansea.ac.uk}



\maketitle

\begin{abstract}
We introduce and study matrix transfers to achieve elementary models for bivariant $K$-theory. They 
share lots of common properties with Voevodsky's framed correspondences and lead to 
symmetric matrix motives of algebraic varieties introduced in this paper. Symmetric matrix motives 
recover $K$-motives and fit in a closed symmetric
monoidal triangulated category of symmetric matrix motives
constructed in this paper by using methods of enriched motivic homotopy theory.
\end{abstract}

\thispagestyle{empty} \pagestyle{plain}

\tableofcontents

\newdir{ >}{{}*!/-6pt/@{>}} 

\section{Introduction}

The role of transfers in Geometry and Topology is multifaceted. If we talk about 
Kasparov $K$-theory or motivic homotopy theory, then an important feature of transfers for $C^*$-algebras or algebraic varieties
is the possibility of computing stable homotopy types associated to various complicated cohomology theories.

We will not dwell on the overview of the achievements of transfers in these areas, but will only 
note the fundamental contribution of Voevodsky who introduced a variety of intricate transfers
for algebraic varieties that led to an abundance of impressive computations in the 
algebro-geometric setting of motivic homotopy theory.

By analogy with Voevodsky's framed transfers of algebraic varieties~\cite{Voe2} we introduce 
matrix transfers for unital (non-commutative) $k$-algebras or (affine) algebraic varieties in this paper. Matrix transfers
share lots of common properties with framed transfers and lead to the concept of
matrix motives, which are constructed in the same fashion with framed motives of algebraic varieties in the sense of~\cite{GP3}. 
Our primary goal is to compute
bivariant $K$-theory by means of matrix motives complementing the existing constructions of bivariant $K$-theories both
for non-commutative algebras and smooth algebraic varieties. 

Stable homotopy theory of spectra is constructed in the same fashion as quasi-coherent sheaves on projective varieties
--- see~\cite{Gark3} for some details. This viewpoint allows to produce spectral categories out of (graded) symmetric 
ring objects using a unified procedure~\cite{GTLMS}. We apply this machinery to our context to introduce a symmetric monoidal 
spectral category $\Mtrc_*(k)$ whose objects are the unital $k$-algebras and symmetric 
spectra of morphisms are defined in terms of matrix transfers. 
Regarding $\Mtrc_*(k)$ as a ``projective variety with several objects", its ``ringoid of global functions" is the category
$\aha$ of unital $k$-algebras and non-unital homomorphisms. The spectral category $\Mtrc_*(k)$ contains much more information than
$\aha$. In particular, its representable modules are stably equivalent to matrix motives after standard algebraic homotopization.

As an application of matrix transfers and matrix motives, 
we introduce and study the closed symmetric triangulated category of symmetric matrix motives $D^\Delta_{\mathrm{mtr}}(k)$ by using enriched 
motivic homotopy theory of spectral categories. This category avoids $\bb A^1$-localization and is realised as the homotopy
category of modules $\Mod\Mtrc_*^\Delta(k)$ over a symmetric monoidal spectral category $\Mtrc_*^\Delta(k)$. The category
$D^\Delta_{\mathrm{mtr}}(k)$ is compactly generated by symmetric matrix motives. Symmetric matrix motives are also shown to
be stably equivalent to $K$-motives in the sense of~\cite{GP,GP1}. In particular, Quillen's $K$-theory of smooth algebraic varieties is 
recovered as the symmetric matrix motive of the point.

We finish the paper by constructing the closed symmetric monoidal triangulated category of big symmetric matrix motives
$D^\Delta_{\mathrm{mtr},\bb G_m}(k)$. We show that $D^\Delta_{\mathrm{mtr}}(k)$ is fully faithfully embedded into
$D^\Delta_{\mathrm{mtr},\bb G_m}(k)$. We also show that $D^\Delta_{\mathrm{mtr},\bb G_m}(k)$ models the category of 
$\kgl$-modules after inverting the exponential characteristic of the base field, where $\kgl$ is a very effective cover of the
algebraic $K$-theory bispectrum $KGL$. We also refer the reader to~\cite{Bac} for another 
model of $\kgl$-modules given in terms of “finite flat correspondences”.

Though major applications of matrix transfers are given for motivic homotopy theory
in this paper, we define and study matrix transfers for all (non-commutative) $k$-algebras
as the author expects similar applications in algebraic Kasparov $K$-theory in the sense of~\cite{Gark,Gark1}.
He does not treat both topics in one go here and primarily concentrates on
applications for motivic homotopy theory.

Throughout the paper $k$ is a fixed commutative ring with unit and
$\ahaw$ is the category of non-unital $k$-algebras
and non-unital $k$-homomorphisms. By $\aha$ we denote the full subcategory in
$\ahaw$ of unital algebras. If there is no likelihood of confusion,
we replace $\otimes_k$ by $\otimes$. If $\cc C$ is a category and
$A,B$ are objects of $\cc C$, we shall often write $\cc C(A,B)$ to
denote the Hom-set $\Hom_{\cc C}(A,B)$. If $\mathcal C$ is a closed symmetric 
monoidal category, we shall sometimes write $[A,C]\in\cc C$
to denote the internal Hom-object $\uhom(A,C)$ associated with 
$A,C\in\mathcal C$.
By geometric realization $|X_{\bullet,\ldots,\bullet}|$ 
of a multi-simplicial set $X_{\bullet,\ldots,\bullet}$ we always mean the diagonal $\diag(X_{\bullet,\ldots,\bullet})$.
We assume 0 to be a natural number.

In general, we shall not be very explicit about set-theoretical
foundations, and we shall tacitly assume we are working in some
fixed universe $\bb U$ of sets. Members of $\bb U$ are then called
{\it small sets\/}, whereas a collection of members of $\bb U$ which
does not itself belong to $\bb U$ will be referred to as a {\it
large set\/} or a {\it proper class}.

\section{The category of matrix transfers $\Mtr_*(\aha)$}\label{sectionmtr}

In this section we introduce the category of matrix transfers on algebras. It is
a kind of a connecting language between bivariant algebraic $K$-theory of algebras and
Voevodsky's theory of framed corrrespondences~\cite{Voe2}. 

\begin{definition}
Given two unital algebras $A,B$ and $n\geq 0$, a {\it matrix transfer of level $n$\/} or a
{\it matrix correspondence of level $n$\/} is a non-unital ring homomorphism
   $$f:A\to M_{2^n}(B),$$
where $M_{2^n}(B)$ is the ring of $(2^n,2^n)$-matrices over the ring $B$.
\end{definition}

We denote by $\Mtr_n(A,B)$ the set of matrix correspondences of 
level $n$ from $A$ to $B$. 
We consider it as a pointed set
pointed at the zero homomorphism $0_n:A\to M_{2^n}(B)$.

As an example, the set $\Mtr_0(A,B)$ is the set of non-unital homomorphisms 
$A\to B$. Note that if $A,B$ are smooth commutative $F$-algebras,
where $F$ is a field, then $\Mtr_0(A,B)$ can be identified with the set of pointed
morphisms $\spec B_+\to\spec A_+$. The latter set coincides with framed correspondences
of level 0 in the sense of Voevodsky~\cite{Voe2}.

Let $f:A\to M_{2^n}(B)$ be an explicit correspondence of level $n$ from $A$
to $B$ and $g:B\to M_{2^m}(C)$ an explicit correspondence of level $m$ from
$B$ to $C$. We define their composition as follows (see~\cite[p.~84]{Gar1} and~\cite[p.~117]{GP1}):
   $$(f:A\to M_{2^n}(B),g:B\to M_{2^m}(C))\longmapsto g\circ f:=M_{2^n}(g)\circ f:A\to M_{2^{n+m}}(C).$$
Here $M_{2^n}(g)$ is the composition of the homomorphism $M_{2^n}(B)\to M_{2^n}M_{2^m}(C)$, induced by $g$ and the natural
isomorphism $M_{2^n}M_{2^m}(C)\cong M_{2^{n+m}}(C)$ obtaned by inserting $(2^m,2^m)$-matrices into entries of a $(2^n,2^n)$-matrix.
The composition of matrix transfers defines associative maps
   \begin{equation*}\label{compos}
    \Mtr_n(A,B)\times \Mtr_m(B,C)\to\Mtr_{n+m}(A,C),\quad n,m\geq 0.
   \end{equation*}

Given a pair of unital algebras $A,B$, denote by $\Mtr_*(A,B)$ the set $\bigsqcup_n\Mtr_n(A,B)$.
Composition of matrix transfers defines a category 
$\Mtr_*(\aha)$. Their objects are those of $\aha$
and morphisms given by $\Mtr_*(A,B)$.
There is an obvious inclusion functor
   $$\aha\to\Mtr_*(\aha)$$
identifying non-unital algebra homomorphisms with matrix correspondences of level 0.

\begin{remark}
The category $\Mtr_*(\aha)$ is reminiscent of the Voevodsky's category of framed correspondences $\Fr_*(Sch/S)$~\cite{Voe2}.
It shares lots of common properties with $\Fr_*(Sch/S)$. In what follows we shall indicate as many common properties as possible.
\end{remark}

Given a unital algebra $A$ and a matrix correspondence
$g:B\to M_{2^{n}}(C)$ of level $n$, we define a matrix correspondence
$\id_A\otimes (g:B\to M_{2^{n}}(C)):A\otimes B\to A\otimes C$ of level $n$ as
   $$A\otimes B\xrightarrow{\id\otimes g}A\otimes M_{2^{n}}(C)\lra\cong M_{2^{n}}(A\otimes C),$$
where the isomorphism is given by the rule $a\otimes(c_{ij})\mapsto(a\otimes c_{ij})$.

Let $f:A_1\to A_2$ be a non-unital homomorphism of of unital algebras. Then the diagram of 
matrix correspondences
   \begin{equation}\label{(6.0.2)}
     \xymatrix{A_1\otimes B\ar[d]_{f\otimes\id_B}\ar[rrr]^{\id_{A_1}\otimes(g:B\to M_{2^{n}}(C))}&&&A_1\otimes C\ar[d]^{f\otimes\id_C}\\
               A_2\otimes B\ar[rrr]_{\id_{A_2}\otimes(g:B\to M_{2^{n}}(C))}&&&A_2\otimes C}
   \end{equation}
commutes. This shows that the category $\Mtr_*(\aha)$ has a ``module" structure over the category
$\aha$ (the latter can also be regarded as $\Mtr_0(\aha)$). 
Given a morphism $f$ in $\aha$ and a matrix correspondence
$g$, one defines $f\otimes g$ to be the matrix correspondence 
given by the diagonal in~\eqref{(6.0.2)}. So we have
an action map
   \begin{equation*}\label{action}
    \aha\boxtimes \Mtr_*(\aha)\to\Mtr_*(\aha).
   \end{equation*}

\begin{remark}
Similarly to the category of framed correspondences $\Fr_*(Sch/S)$ of Voevodsky~\cite{Voe2},
the category $\Mtr_*(\aha)$ has neither an initial nor a final object.
Endomorphisms of the zero algebra in $\Mtr_*(\aha)$ are of the form $\Mtr_*(0,0)=
\{0_0,\ldots,0_n,\ldots\}$. Observe also that the direct product of algebras is not their
product in $\Mtr_*(\aha)$.
\end{remark}

\begin{definition}
A {\it $\Mtr_*$-functor\/} or a {\it matrix functor\/} $\cc F$ on $\aha$ is a covariant
functor from $\Mtr_*(\aha)$ to the
category of sets. We say that $\cc F$ is {\it additive\/} if it lands in the
category of pointed sets such that $\cc F(0)=pt$ and $\cc
F(A\times B)=\cc F(A)\times\cc F(B)$.
\end{definition}

Note that the representable functors on $\Mtr_*(\aha)$ are not additive.
To associate an additive functor to an algebra, one needs the following construction. Denote by
$\sigma_A$ the matrix correspondence of level 1 from $A$ to $A$
given by the non-unital homomorphism $A\to M_2(A)$ sending $A$ to the upper left conner of $M_2(A)$.
For any non-unital homomorphism $f:A\to B$ one has $f\sigma_A=\sigma_B f$.
However, for a general matrix correspondence $f$ one
has $f\sigma_A\ne\sigma_B f$. 
We set
   $$\Mtr(A,-):=\colim(\Mtr_0(A,-)\xrightarrow{\sigma_A}\Mtr_1(A,-)\xrightarrow{\sigma_A}\Mtr_2(A,-)\xrightarrow{\sigma_A}\cdots).$$
Then $\Mtr(A,-)$ is a covariant functor from $\Mtr_*(\aha)$ to sets. 

\begin{lemma}\label{ka}
For any $A\in\aha$ the functor $\Mtr(A,-)$ is additive. Moreover, $\Mtr(A,B)=\Hom_{\ahaw}(A,M_\infty (B))$,
where $B\in\aha$ and $M_\infty(B)=\bigcup_nM_n(B)$ is the union of all finite matrices over $B$.
\end{lemma}

\begin{proof}
First observe that $\Mtr_*(A,0)=\{0_0,\ldots,0_n,\ldots\}$ and that composition with $\sigma_A$
takes $0_n$ to $0_{n+1}$. As a result, $\sigma_A$ glues the distinguished morphisms $0_n$-s,
and hence the functor $\Mtr(A,-)$ lands in pointed sets with $\Mtr(A,0)=pt$. Suppose $f\in\Mtr(A,B)$
and $g\in\Mtr(A,B')$. Applying $\sigma_A$ if necessary, we may assume that $f,g$ are represented by two matrix correspondences
of the same level $n$, say. They give a matrix correspondence $f\times g:A\to M_{2^n}(B\times B')$ and
the rule $(f,g)\mapsto f\times g$ plainly yields a bijection $\Mtr(A,B)\times\Mtr(A,B')\cong\Mtr(A,B\times B')$.
The formula $\Mtr(A,B)=\Hom_{\aha}(A,M_\infty (B))$ is obvious.
\end{proof}

\begin{definition}
A matrix functor $\cc F$ is said to be {\it $\sigma$-invariant\/}
(respectively {\it stable}) if for any $A\in\aha$ one has $\cc
F(\sigma_A)$ is an isomorphism (respectively $\cc
F(\sigma_A)=\id_{\cc F(A)}$). We say that $\cc F$ is {\it homotopy invariant\/}
if takes $A\to A[t]$, $A\in\aha$, to an isomorphism.
\end{definition}

The notion of a $\sigma$-invariant/stable functor is similar to that
of a (quasi-)stable framed presheaf (see~\cite{Voe2,GP3}). Note that
the matrix functor $\Mtr(A,-)$, $A\in\aha$, is not
$\sigma$-invariant.

Given $A\in\aha$ let $\sigma'_A\in\Mtr_1(A,A)$ be the
correspondence $a\mapsto\left(\begin{array}{cc} 0 & 0\\ 0 & a \\
\end{array}\right)$. It is useful to have the following result.

\begin{lemma}\label{m2stableshtrih}
$\cc F(\sigma_A)=\cc F(\sigma_A')$ for every $\sigma$-invariant
matrix functor $\cc F$ and algebra $A\in\aha$.
\end{lemma}

\begin{proof}
This follows from~\cite[Lemma~2.2.4]{Cor}.
\end{proof}

There is a distinguished level one matrix correspondence for every
$A\in\aha$
   $$\diag:A\times A\to A,\quad (a_1,a_2)\mapsto\left(\begin{array}{cc} a_1 & 0
   \\ 0 & a_2 \\ \end{array}\right).$$
It is similar to the level one framed correspondence $\delta\in
Fr_1(X,X\sqcup X)$, $X\in Sch/S$, defined by the triple $(U=(\bb
A^1-\{0\}\sqcup\bb A^1-\{1\})_X,\phi=(t-1)\sqcup t,g:(\bb
A^1-\{0\}\sqcup \bb A^1-\{1\})_X\to X\sqcup X),$  where $t:\bb
A^1_X\to X$ is the projection (see~\cite{Voe2}).

Any additive matrix functor $\cc F$ defines maps
$$\cc F(A)\times\cc F(A)= \cc F(A\times A)\xrightarrow{\diag_*} \cc F(A).$$

The following statement is reminiscent of Voevodsky's
theorem~\cite[3.6]{Voe2} for additive stable homotopy invariant
framed presheaves.

\begin{theorem}\label{salavat}
Let $\cc F$ be an additive stable homotopy invariant matrix 
functor. Then the maps $\cc F(A)\times\cc F(A)\to\cc F(A)$ make
$\cc F$ into a matrix functor of abelian monoids. The neutral
element of the monoid is determined by the morphism $pt=\cc
F(0)\to\cc F(A)$ induced by a unique morphism $0\to A$ of level
zero.
\end{theorem}

\begin{proof}
For an additive matrix functor $\cc F$ and a pair of correspondences
$g_1,g_2$ of the same level we have $\cc F(g_1\times g_2)=\cc
F(g_1)\times\cc F(g_2)$, where $g_1\times g_2$ is defined in the
obvious manner. Therefore, to check that our operation is
associative, we need to check that the following diagram in 
$\Mtr_*(\aha)$ commutes up to polynomial homotopy:
   \begin{equation}\label{kzvezda}
     \xymatrix{A\times A\times A\ar[rr]^{\sigma_A\times\diag}\ar[d]_{\diag\times\sigma_A}&&A\times A\ar[d]^{\diag}\\
               A\times A\ar[rr]_{\diag}&&A.}
   \end{equation}
One has,
   $$\diag(\sigma_A\times\diag)(a_1,a_2,a_3)=\left(\begin{array}{cccc}
                                          a_1 & 0 & 0 & 0 \\
                                          0 & a_2 & 0 & 0 \\
                                          0 & 0 & 0 & 0 \\
                                          0 & 0 & 0 & a_3 \\
                                        \end{array}\right)$$
and
   $$\diag(\diag\times\sigma_A)(a_1,a_2,a_3)=\left(\begin{array}{cccc}
                                          a_1 & 0 & 0 & 0 \\
                                          0 & a_3 & 0 & 0 \\
                                          0 & 0 & a_2 & 0 \\
                                          0 & 0 & 0 & 0 \\
                                        \end{array}\right)$$
Let $f\in\Mtr_2(A\times A\times A,A)$ be defined as
   $$f(a_1,a_2,a_3)=\left(\begin{array}{cccc}
                                          a_1 & 0 & 0 & 0 \\
                                          0 & 0 & 0 & 0 \\
                                          0 & 0 & a_2 & 0 \\
                                          0 & 0 & 0 & a_3 \\
                                        \end{array}\right)$$
We claim that $\diag(\sigma_A\times\diag)$ and
$\diag(\diag\times\sigma_A)$ are polynomially homotopic to $f$. We
use rotational homotopy. Namely, consider any matrix $T=\left(
\begin{array}{cc}
t_{11}(x)&t_{12}(x)\\t_{21}(x)&t_{22}(x)\end{array} \right)\in
GL_2(k[x])$ with
   $$\partial^0_x(T)=\left( \begin{array}{cc} 1&0 \\0&1 \end{array} \right) \quad \mbox{and} \quad
     \partial^1_x(T) = \left( \begin{array}{cc} 0&-1 \\1&0 \end{array}\right).$$
The matrix
   \begin{equation}\label{rot}
    T=\left(\begin{array}{cc}1-x^2
    &x^3-2x\\x&1-x^2\end{array}\right)\in SL_{2}(k[x])
   \end{equation}
is a concrete example over any with
   $$T^{-1}=\left(\begin{array}{cc}1-x^2 &2x-x^3\\-x&1-x^2\end{array}\right).$$
Set,
   $$C=\left(\begin{array}{cccc}          1 & 0 & 0 & 0 \\
                                          0 & t_{11}(x) & t_{12}(x) & 0 \\
                                          0 & t_{21}(x) & t_{22}(x) & 0 \\
                                          0 & 0 & 0 & 1 \\
                                        \end{array}\right),\quad
     D=\left(\begin{array}{cccc}
                                          1 & 0 & 0 & 0 \\
                                          0 & t_{11}(x) & 0 & t_{12}(x) \\
                                          0 & 0 & 1 & 0 \\
                                          0 & t_{21}(x) & 0 & t_{22}(x) \\
                                        \end{array}\right)$$
Then $C,D\in GL_2(k[x])$ and $C^{-1}(\diag(\sigma_A\times\diag))
C,D^{-1}(\diag(\diag\times\sigma_A))D\in\Mtr_2(A\times A\times
A,A[x])$ give the desired polynomial homotopies
$\diag(\diag\times\sigma_A)\sim f$ and
$\diag(\diag\times)\sigma_A\sim f$. We see that~\eqref{kzvezda} is a
commutative square up to polynomial homotopy.

Likewise, by using rotational homotopies, one can easily show that
the diagram
   $$\xymatrix{A\times A\ar[r]^(.45){swap}\ar[d]_{\diag}&A\times A\ar[r]^(.55){\diag}&A\ar[d]^{\sigma_A}\\
               A\ar[rr]_{\sigma_A}&&A}$$
is commutative up to polynomial homotopy. Thus $\cc F(A)$ is an
Abelian semigroup. Since $(A\cong A\times 0\xrightarrow{\diag}
A)\in\Mtr_1(A,A)$ equals $\sigma_A$ and $(A\cong 0\times
A\xrightarrow{\diag} A)\in\Mtr_1(A,A)$ equals $\sigma_A'$,
Lemma~\ref{m2stableshtrih} implies $\cc F(A)$ is an Abelian monoid
with the neutral element being determined by the morphism $pt=\cc
F(0)\to\cc F(A)$ induced by a unique morphism $0\to A$ of level zero.
\end{proof}

The rotational matrix $T$ above~\eqref{rot} plays an important role in our analysis.
It is a product of elementary matrices from $SL_2(k[x])$:
   $$T=\left(\begin{array}{cc}1&-x\\0&1\end{array}\right)
       \left(\begin{array}{cc}1&0\\x&1\end{array}\right)
       \left(\begin{array}{cc}1&-x\\0&1\end{array}\right)$$
We will also need the following lemma.

\begin{lemma}\label{barbashev}
Let $\Sigma_n\to GL_n(k)$ be the standard inclusion and let $\tau$ be an even
permutation. Then there is a matrix $L(x)\in SL_n(k[x])$
such that $L(0)=I$ is the identity matrix and $L(1)=\tau.$
\end{lemma}

\begin{proof}
It is enough to show the statement for $k=\bb Z$. The proof is like that of~\cite[Lemma~2.10]{GN}.
Since $\tau$ is even, its image belongs to $SL_n(\bb Z)$. Since
$SL_n(\bb Z)$ is generated by elementary matrices of the form $e_{i,j}(\lambda)$
by~\cite[Exercise~III.1.5]{Wei}
(all its diagonal elements equal to 1, 
$\lambda$ is placed in the $(i,j)$-th entry and zero elsewhere), $\tau$
can be written as a product of elementary matrices:
\[\tau = \prod_{l=1}^m e_{i_l, j_l}(\lambda_l),
\quad\text{ where } 1\leqslant i_l,j_l\leqslant n,\, \lambda_l\in k,\, i_l\neq j_l.\] 
Then $L(x)=\prod_{l=1}^m e_{i_l, j_l}(\lambda_l x)$ defines a matrix
in $SL_n(k[x])$ with $L(0)=I$ and $L(1)=\tau.$
\end{proof}

For a matrix functor $\cc F$ denote by $C_*\cc F$
the simplicial matrix functor of the form $B\mapsto\cc
F(B^{\Delta})$, where $B^{\Delta}$ is the standard simplicial object
in $\aha$
   $$n\longmapsto B^{\Delta^n}=B[x_0,\ldots,x_n]/(x_0+\cdots+x_n-1).$$
For any $k$-algebra $A\in\aha$ consider the simplicial functor taking
$B\in\aha$ to the simplicial set $C_*(\Mtr(A,-))(B)=\Mtr(A,B^{\Delta})$.
As $\sigma_B$ commutes with the
morphisms of level zero, the matrix transfers $\sigma:B^{\Delta^n}\to
B^{\Delta^n}$ of level one define an endomorphism of $C_*(\Mtr(A,-))(B)$ which
we denote $\sigma$. An endomorphism $\sigma'$ of 
$C_*\Mtr(A,-)(B)$ is defined in a similar way by the morphisms
$\sigma_B'$.

In what follows we shall write $\pi_\ell(C_*\Mtr(A,-))$ to denote the matrix functor 
   $$B\in\aha\mapsto\pi_\ell(C_*\Mtr(A,B),0),$$
where the homotopy groups are taken for the simplicial sets $C_*\Mtr(A,B)$ pointed at the zeroth homomorphism.
The main example of a stable homotopy invariant matrix functor is
given by the $\ell$th homotopy group $\pi_\ell(C_*\Mtr(A,-))$. To
see this, let us prove the following statement
(cf.~\cite[Lemma~6.1]{Voe2}):

\begin{lemma}\label{pin}
For any matrix correspondence $f\in\Mtr_n(A,B)$ the matrix 
cor\-re\-spondences $f\circ\sigma_A$ and $\sigma_B\circ f$ are
polynomially homotopic. Likewise, the matrix 
cor\-re\-spondences $f\circ\sigma_A'$ and $\sigma_B'\circ f$ are
canonically polynomially homotopic.
\end{lemma}

\begin{proof}
For every element $a\in A$ the matrix $f\circ\sigma_A(a)$ has entries 
  $$(b_{1,1}\ldots b_{1,2^{n}}\ 0\bl{2^{n}}\ldots0)$$
on the first row (for simplicity we display the first row only). The 
matrix $\sigma_B\circ f(a)$ has entries 
  $$(b_{1,1}0\ b_{1,2}\ 0\ldots 0\ b_{1,2^{n}}\ 0)$$
on the first row. We can apply $n$ disjoint cycles $\tau=(243)(687)\cdots(2^{n+1}-2\ 2^{n+1}\ 2^{n+1}-1)$ 
of length 3 to the latter matrix to get a matrix with entries
  $$(b_{1,1}\ b_{1,2}\ 0\ 0\ldots b_{1,2^n-1} \ b_{1,2^n}\ 0\ 0)$$
on the first row. Note that $\tau$ is an even permutation. We can then
apply an even permutation to the latter matrix (it successively permutes two zeros on the left with the nearest couple $b_{1,i-1},b_{1,i}$) 
to get the matrix $f\circ\sigma_A(a)$.
We see that the matrix $f\circ\sigma_A(a)$ is conjugate to the matrix $\sigma_B\circ f(a)$
by an even permutation matrix for any $a\in A$.

Using Lemma~\ref{barbashev} we conclude that there is $Q\in
SL_{2^{n+1}}(k[x])$ such that the composition
   $$h:A\xrightarrow{\sigma_B\circ f}M_{2^{n+1}}(B)\to Q^{-1}M_{2^{n+1}}(B) Q$$
has the property that $\partial^0_x(h)=f\circ \sigma_A$ and
$\partial^1_x(h)=\sigma_B\circ f$. Thus $f\circ\sigma_A$ and $\sigma_B\circ f$ are
polynomially homotopic, as required. Similarly, the matrix 
cor\-re\-spondences $f\circ\sigma_A'$ and $\sigma_B'\circ f$ are
polynomially homotopic.
\end{proof}

\begin{lemma}\label{magnitka}
$\pi_\ell(C_*\Mtr(A,\sigma))=\pi_\ell(C_*\Mtr(A,\sigma'))=\id$ for every $\ell\geq 1$. In
particular, each matrix functor $\pi_\ell(C_*\Mtr(A,-))$ is
additive, stable and homotopy invariant.
\end{lemma}

\begin{proof}
It is enough to prove the lemma for the map $\sigma$, because the
same arguments are used for the map $\sigma'$. The proof of
Lemma~\ref{pin} shows that there is a map
   $$h_n:\Mtr_n(A,B^{\Delta})\to\Mtr_{n+1}(A,B^{\Delta}[x]),\quad n\geq 0,$$
such that $(\partial^0_xh_n)(f)=f\circ \sigma_A$ and
$(\partial^1_xh_n)(f)=\sigma_B\circ f=\sigma(f)$. The homotopy $h_n$ is 
constructed explicitly by using permutation matrices associated with even permutations and
Lemma~\ref{barbashev}. It follows
from~\cite[Proposition~3.2]{Gar} that $\sigma_B:\Mtr_n(A,B^{\Delta})\to\Mtr_{n+1}(A,B^{\Delta})$ is simplicially homotopic to the map
$\sigma_A:\Mtr_n(A,B^{\Delta})\to\Mtr_{n+1}(A,B^{\Delta})$.
Moreover, the homotopies $h_n$ fit in a commutative diagram
   $$\xymatrix{\cdots\ar[r]^(.3){\sigma_A}&\Mtr_n(A,B^{\Delta})\ar[d]_{h_n}\ar[r]^{\sigma_A}
               &\Mtr_{n+1}(A,B^{\Delta}))\ar[d]^{h_{n+1}}\ar[r]^(.7){\sigma_A}&\cdots\\
               \cdots\ar[r]^(.3){\sigma_A}&\Mtr_{n+1}(A,B^{\Delta}[x])\ar[r]^{\sigma_A}
               &\Mtr_{n+2}(A,B^{\Delta}[x])\ar[r]^(.7){\sigma_A}&\cdots}$$
We tacitly use here the fact that homotopies produced by elementary matricies of Lemma~\ref{barbashev}
applied to zero rows/columns of a matrix do not change the matrix.
We also use the fact that $e_{ij}(r)e_{kl}(s)=e_{kl}(s)e_{ij}(r)$ if $i\not= k$ and $j\not= l$~\cite[p.~201]{Wei}.
Thus one gets a homotopy 
   $$h:\Mtr(A,B^{\Delta})\to\Mtr(A,B^{\Delta}[x])$$
of simplicial sets such that $(\partial^0_xh)=\id$ and
$(\partial^1_xh)=\sigma_B$. This homotopy leaves the zeroth homomorphism stationary. Thus
$\pi_\ell(C_*\Mtr(A,\sigma))=\id$ for every $\ell\geq 1$.
\end{proof}

The proof of Lemma~\ref{magnitka} shows the following useful fact.

\begin{corollary}\label{traktor}
The maps
   $\sigma,\sigma':\Mtr(A,B^{\Delta})\to\Mtr(A,B^{\Delta})$
are simplicially homotopic to the identity.
\end{corollary}

We finish the section by proving the following

\begin{theorem}\label{slovan}
For every $A,B\in\aha$ the map
   $$\Mtr(A,B^{\Delta})\times\Mtr(A,B^{\Delta})\cong\Mtr(A,(B\times B)^{\Delta})
     \xrightarrow{\diag}\Mtr(A,B^{\Delta})$$
determines the structure of a homotopy associative, homotopy
commutative $H$-space. This structure is functorial in matrix
correspondences of level zero in both arguments. In particular,
positive homotopy groups $\pi_{\ell\geq 1}(\Mtr(A,B^{\Delta}))$ are
Abelian with group structure being induced by the $H$-space
structure.
\end{theorem}

\begin{proof}
By Corollary~\ref{traktor} the maps
$\sigma,\sigma'$ are simplicially homotopic to
the identity. Applying $C_*\Mtr(A,-)$ to the homotopy
commutative diagrams in the proof of Theorem~\ref{salavat} and using
the fact that $C_*\Mtr(A,-)$ converts polynomial
homotopies to simplicial ones, we obtain a homotopy associative,
homotopy commutative $H$-space structure on $C_*\Mtr(A,-)$. 
Its functoriality in both arguments is obvious.

Finally, $\pi_{1}(\Mtr(A,B^{\Delta}))$ is Abelian, because
$\Mtr(A,B^{\Delta})$ is an $H$-space~\cite[9.31]{Strom}.
The fact that the group structure on each $\pi_{\ell\geq 1}(\Mtr(A,B^{\Delta}))$ 
coincides with the one induced by the $H$-space
structure follows from the Eckmann--Hilton
argument~\cite[9.30]{Strom}.
\end{proof}

\section{The spectral category $\Mtrc_*(k)$}\label{sectionmtrc}

In order to construct a spectral category associated with matrix correspondences,
we use results of~\cite{GTLMS}. The idea is that matrix correspondences (similarly to Voevodsky's framed 
correspondences) have more structures than just morphisms of a category. 
We start with preparations.

Following~\cite[Section~7]{H}, let $\Sigma$ be the category whose objects are the natural numbers.
The morphisms of $\Sigma$ are the symmetric groups $\Sigma_n$, $n\geq 0$.
Given a category $\cc C$, a {\it symmetric sequence in $\cc C$} is a functor $\Sigma\to\cc C$. 
A symmetric sequence $X$ in a category $\cc C$ can also be regarded as a sequence $(X_0,X_1,X_2,\ldots)$ of objects
of $\cc C$ with an action of $\Sigma_n$ on $X_n$.
The {\it category of symmetric sequences\/} $\Sigma\cc C$ is the functor category $\Fun(\Sigma,\cc C)$.

Consider the dual category $\Alg_k^{\op}$ of $\Alg_k$.
It is symmetric monoidal with finite coproducts and monoidal unit $k\in\Alg_k^{\op}$. The monoidal product
(respectively coproduct) of $A,B\in\Alg_k^{\op}$, denoted by $A\wedge B$
(respectively $A\vee B$), is given by $A\otimes_k B$ (respectively $A\times B$). Then
a canonical morphism
   $$v:\bigvee_{i\in I}(A_i\wedge B)\to(\bigvee_{i\in I}A_i)\wedge B$$
is an isomorphism for any finite set $I$ and $A_i,C\in\Alg_k^{\op}$. In particular, if $I=\emptyset$ then $0\wedge B=B\wedge 0=0$.
As the category $\Alg_k$ has finite limits  and zero object, its dual category $\Alg_k^{\op}$ has finite colimits  and zero object.

By~\cite[Section~7]{H} the category of symmetric sequences
$\Sigma\Alg_k^{\op}$ in the sense of~\cite[Definition~7.1]{H} is symmetric monoidal with
   $$(X\wedge Y)_n=\bigvee_{p+q=n}\Sigma_n\times_{\Sigma_p\times\Sigma_q}X_p\wedge Y_q.$$
The symmetric sequence $(k,0,0,\ldots)$ is a monoidal unit of $\Sigma\Alg_k^{\op}$.
This notation needs some explanation (we follow~\cite[Section~7]{H}).
Given a finite set $\Gamma$ and an object $A\in\Alg_k^{\op}$, $\Gamma\times A$ is
the coproduct of $|\Gamma|$ copies of $A$. If $\Gamma$ is a group, then $\Gamma\times A$ has an obvious left
$\Gamma$-action; $\Gamma\times A$ is the free $\Gamma$-object on $A$.
Note that a $\Gamma$-action on $A$ is then equivalent to a map $\Gamma\times A\to A$
satisfying the usual unit and associativity conditions. Also, if $\Gamma$ admits a right action by a
group $\Gamma'$, and $A$ is a left $\Gamma'$-object, then we can form $\Gamma\times_{\Gamma'} A$
as the colimit of the $\Gamma'$-action on $\Gamma\times A$,
where $\alpha\in\Gamma'$ takes the copy of $A$ corresponding to
$\beta\in\Gamma$ to the copy of $A$ corresponding to $\beta\alpha^{-1}$ by the action of $\alpha$. 

Denote by $\Aff_k$ the full subcategory of $\Alg_k^{\op}$ consisting of the unital
$k$-algebras. By definition, $\Aff_k^{\op}=\Alg_k^u$. Note that $\Aff_k$ is closed under $\wedge$.
Let $E$ be a ring object in $\Sigma\Alg_k^{\op}$ and let $P$ be an object of $\Alg_k^{\op}$.
Following~\cite[Section~2]{GTLMS} we define the set of {\it $(E,P)$-correspondences of
level $n$\/} between two objects $X,Y\in\Aff_k$ by
   $$\corr^E_n(X,Y):=\Hom_{\Alg_k^{\op}}(X\wedge P^{\wedge n},Y\wedge E_n).$$
This set is pointed at the zeroth map. By definition, $\corr^E_0(X,Y):=\Hom_{\Alg_k^{\op}}(X,Y\wedge E_0)$.
The set $\corr^E_n(X,Y)$ is defined similarly to the set of Voevodsky's framed correspondences $\Fr_n(X,Y)$ 
of level $n$ (see~\cite{GTLMS} for details).

Define a pairing
   $$\phi_{X,Y,Z}:\corr^E_n(X,Y)\wedge\corr^E_m(Y,Z)\to\corr^E_{n+m}(X,Z)$$
by the rule: $\phi_{X,Y,Z}(f:X\wedge P^{\wedge n}\to Y\wedge E_n,g:Y\wedge P^{\wedge m}\to Z\wedge E_m)$
is given by the composition
\begin{gather*}
X\wedge P^{\wedge n}\wedge P^{\wedge m}\xrightarrow{f\wedge P^{\wedge m}}
Y\wedge E_n\wedge P^{\wedge m}
\xrightarrow{tw} Y\wedge P^{\wedge m}\wedge E_n\xrightarrow{g\wedge E_n} Z\wedge E_{m}\wedge E_n\xrightarrow{tw}\\
\to Z\wedge E_n\wedge E_m\xrightarrow{Z\wedge \mu_{n,m}}Z\wedge E_{n+m}.
\end{gather*}
By~\cite[Theorem~2.4]{GTLMS} $\Aff_k$ is enriched over the closed symmetric monoidal 
category of symmetric sequences of pointed sets $\Sigma Sets_*$.
Namely, $\Sigma Sets_*$-objects of morphisms are defined by
   \begin{equation}\label{skobki}
    (\corr^E_0(X,Y),\corr^E_1(X,Y),\corr^E_2(X,Y),\ldots),\quad X,Y\in\Aff_k.
   \end{equation}
Compositions are defined by
pairings $\phi_{X,Y,Z}$. The resulting $\Sigma Sets_*$-category is denoted by $\corr^E_*(\Aff_k)$.
Moreover, if $E$ is a commutative ring object then $\corr^E_*(\Aff_k)$ is symmetric monoidal. The
monoidal product of $X,Y\in\Ob(\corr^E_*(\Aff_k))$ is defined as $X\wedge Y$, where $\wedge$ is
the monoidal product in $\Aff_k$.

The left action of the symmetric group $\Sigma_n$ on 
$\corr^E_n(X,Y)=\Hom_{\Alg_k^{\op}}(X\wedge P^{\wedge n},Y\wedge E_n)$ for each $n\geq 0$ is given by
conjugation. In detail, for each $f:X\wedge P^{\wedge n}\to Y\wedge E_n$
and each $\tau\in\Sigma_n$ the morphism $\tau\cdot f$ is defined as the composition
   \begin{equation}\label{sigmaseq}
    X\wedge P^{\wedge n}\xrightarrow{X\wedge\tau^{-1}}X\wedge P^{\wedge n}\xrightarrow{f}Y\wedge E_n
       \xrightarrow{Y\wedge\tau}Y\wedge E_n.
   \end{equation}
With this definition each $\corr^E_n(X,Y)$ becomes a pointed $\Sigma_n$-set. Here $\Sigma_n$
acts on $P^{\wedge n}$ by permutations, using the commutativity and associativity isomorphisms.

We are now in a position to prove the following result.

\begin{theorem}\label{mtrsigma}
The category $\Alg^u_k$ can be enriched over the closed symmetric monoidal 
category of symmetric sequences of pointed sets $\Sigma Sets_*$ by means of matrix transfers.
Namely, $\Sigma Sets_*$-objects of morphisms are defined by
   \begin{equation}\label{skobkimtr}
    (\Mtr_0(A,B),\Mtr_1(A,B),\Mtr_2(A,B),\ldots),\quad A,B\in\Alg^u_k.
   \end{equation}
Compositions are defined by compositions of matrix transfers. This $\Sigma Sets_*$-category is symmetric monoidal
with respect to the tensor product of algebras. It
will be denoted by $\Mtr_*^\Sigma(k)$ and called the \emph{category of symmetric matrix transfers}.
\end{theorem}

\begin{proof}
Consider $P=M_2(k)\in\Alg_k^{\op}$ and $E=(k,k,k,\ldots)$. Then $E$ is a commutative ring object in
$\Sigma\Alg_k^{\op}$ with trivial action of the symmetric groups. By above we get a symmetric 
monoidal $\Sigma Sets_*$-category of $(E,P)$-correspondences
$\corr_*^E(\Aff_k)$ on $\Aff_k$.

By~\cite[Proposition~6.2.2]{Bor} we can associate to $\corr_*^E(\Aff_k)$ its dual $\Sigma Sets_*$-category
of correspondences $\corr_*^{E,\op}(\Aff_k)$, which is plainly symmetric monoidal. 
The object
$P^{\wedge n}=M_2(k)\otimes\bl n\cdots\otimes M_2(k)$ is 
naturally isomorphic to $M_{2^n}(k)$ in two different ways. The first (respectively second) 
isomorphism inserts $2\times 2$-matrices from left to right (respectively from right to left). 
The action of the symmetric group $\Sigma_n$ on $M_{2^n}(B)$ is induced from its action on $P^{\wedge n}$.
We permute factors on $P^{\wedge n}$ and apply the first isomorphism onto $M_{2^n}(k)$ 
that inserts $2\times 2$-matrices into 
$2\times 2$-matrix from left to right. In turn, we insert $2\times 2$-matrices from right to left in matrix transfers.
By construction, $\tau\in\Sigma_n$ acts on $P^{\wedge n}$ by $\tau^{-1}$ in $\corr_*^{E,\op}(\Aff_k)$,
where factors on $P^{\wedge n}$ are enumerated from left to right.
It acts on $M_{2^n}(k)$ as $\tau^{-1}$ in $\Mtr_*^\Sigma(k)$ as well. In more detail, 
permutations act on $M_2(k)\otimes\bl n\ldots\otimes M_2(k)$ as usual by permuting its factors 
except that we enumerate the factors from right to left.
Afterwards we apply the second isomorphism $M_2(k)\otimes\bl n\ldots\otimes M_2(k)\lra{\cong}M_{2^n}(k)$
and the canonical isomorphism $M_{2^n}(k)\otimes B\cong M_{2^n}(B)$.
Then $(f:A\to M_{2^n}(B))\in\Mtr_n(A,B)$ 
is mapped to the composition
   $$A\lra f M_{2^n}(B)\lra{\tau^{-1}}M_{2^n}(B).$$
It remains to observe that
$\Mtr_*^\Sigma(k)$ is isomorphic to $\corr_*^{E,\op}(\Aff_k)$. 
\end{proof}

Let $\Gamma$ be the Segal category of finite sets~\cite{Seg}.
Recall that a morphism from $S$ to $T$ in $\Gamma$ is a set map $\theta$ from $S$
to the power set $\cc P(T)$ of $T$ such that $\theta(i)$ and $\theta(j)$ are
disjoint for any $i\not=j$ in $S$. Recall that its opposite category $\Gamma^{\op}$
is the category of finite pointed sets and pointed maps.

Given $S\in\Gamma$ and $A\in\aha$, we write $A\odot S$ for $\prod_{S}A$, the direct product of
$S$ copies of $A$. A map $\theta:S\to T$ in $\Gamma$ induces a non-unital
homomorphism $\theta_*:A\odot{S}\to A\odot T$ defined as $\theta_*(a_i)_{i\in S}=(b_j)_{j\in T}$
with $b_j=a_i$ if $j\in{\theta(i)}$ and $b_j=0$ if $j\not\in\cup_{i\in S}\theta(i)$.
Given a pointed simplicial set $L$ such that the pointed set
of $n$-simplicies is finite for every $n\geq 0$, we can regard it as a cosimplicial
object in $\Gamma$. We write $A\odot L$ to denote the induced cosimplicial
object in $\aha$.

We can now pass to the construction of a symmetric monoidal spectral category $\Mtrc_*(k)$.
It is similar to the construction of $\Mtr_*^\Sigma(k)$. The only difference is that we add simplicial
spheres in each degree of spectral morphisms between algebras. More precisely, for any $n\geq 0$ set
   $$\Mtrc_n(A,B)=\Mtr_n(A\odot S^n,B),\quad A,B\in\aha.$$
Next we set,
   $$\Mtrc_*(A,B)=(\Mtrc_0(A,B),\Mtrc_1(A\odot S^1,B),\Mtrc_2(A\odot S^2,B),\ldots).$$
The left action of $\Sigma_n$ on
$\Mtrc_n(A,B)$ is given by conjugation: for each $f:A\odot S^n\to M_{2^n}(B)$
and each $\tau\in\Sigma_n$ the morphism $\tau\cdot f$ is defined as
the composition
   \begin{equation}\label{sigmaseqq}
    A\odot S^n\xrightarrow{A\odot\tau}A\odot S^n\xrightarrow{f}M_{2^n}(B)
       \xrightarrow{\tau^{-1}}M_{2^n}(B).
   \end{equation}
We recall the reader that we numerate the $2\times 2$-blocks of $M_{2^n}(B)$ from right to left
(see the proof of Theorem~\ref{mtrsigma} for details).

Each structure map
   $$\Mtrc_m(A,B)\wedge S^1\to\Mtrc_{m+1}(A,B)$$
is defined as follows. The algebra of $\ell$-cosimplices $(A\odot S^{m+1})_\ell$ is, by construction,
$\prod_{x\in S^1_\ell}(A\odot S^m)_\ell$. Also, $\Mtr_m(A\odot S^m,B)_\ell\wedge S^1_\ell=\bigvee_{x\in S^1_\ell}\Mtr_m(A\odot S^m,B)_\ell$.
Given a homomorphism $f:(A\odot S^m)_\ell\to M_{2^m}(B)$ in the summand $\Mtr_m(A\odot S^m,B)_\ell$ with index $x\in S^1_\ell$,
we map it to an element of $\Mtrc_{m+1}(A\odot S^{m+1},B)_\ell$ defined as the composition
   $$(A\odot S^{m+1})_\ell=\prod_{x\in S^1_\ell}(A\odot S^m)_\ell\xrightarrow{proj_x}(A\odot S^m)_\ell\xrightarrow{f}M_{2^m}(B)
       \xrightarrow{M_{2^{m}}(\sigma_B)}M_{2^{m+1}}(B).$$
The structure maps induce natural $(\Sigma_m\times\Sigma_k)$-equivariant maps 
   $$\lambda_{m,k}:\Mtrc_m(A,B)\wedge S^k\to\Mtrc_{m+k}(A,B),$$
so that $\Mtrc_*(A,B)$ becomes a symmetric $S^1$-spectrum.

In order to construct pairings
   \begin{equation}\label{avtomobilist}
    \Mtrc_*(A,B)\wedge\Mtrc_*(B,C)\to\Mtrc_*(A,C),\quad A,B,C\in\aha,
   \end{equation}
recall from~\cite[Section~I.5.1]{Sch} that a map of symmetric spectra $X\wedge Y\to Z$, $X,Y,Z\in Sp^\Sigma$,
is uniquely determined by a bimorphism $b:(X,Y)\to Z$ that consists of a collection of $\Sigma_p\times\Sigma_q$-equivariant maps
$b_{p,q}:X_p\wedge Y_q\to Z_{p+q}$, for $p,q\geq 0$, such that the ``bilinearity diagram"
   \begin{equation}\label{yugra}
    \xymatrix{&&X_p\wedge Y_q\ar[lld]_{X_p\wedge\sigma_q}\wedge S^1\ar[d]^{b_{p,q}\wedge S^1}\ar[rr]^{X_p\wedge\textrm{twist}}&&X_p\wedge S^1\wedge Y_q\ar[d]^{\sigma_p\wedge Y_q}\\
              X_p\wedge Y_{q+1}\ar[drr]_{b_{p,q+1}}&&Z_{p+q}\wedge S^1\ar[d]^{\sigma_{p+q}}&&X_{p+1}\wedge Y_q\ar[d]^{b_{p+1,q}}\\
              &&Z_{p+q+1}&&Z_{p+1+q}\ar[ll]_{1\times\chi_{1,q}}}
   \end{equation}
commutes for all $p,q\geq 0$. Here $\chi_{1,q}\in\Sigma_{1+q}$ is the shuffle permutation.

The desired pairing~\eqref{avtomobilist} is determined by a bimorphism that consists of a collection of
$\Sigma_p\times\Sigma_q$-equivariant maps
   \begin{equation}\label{bim}
    b_{p,q}:\Mtrc_p(A,B)\wedge\Mtrc_q(B,C)\to\Mtrc_{p+q}(A,C)
   \end{equation}
taking $(A\odot S^p\lra{f} M_{2^p}(B),B\odot S^q\lra{g} M_{2^q}(C))$ to the composition
   $$A\odot S^{p+q}\xrightarrow{f\odot S^q}M_{2^p}(B)\odot S^q\bl{can}\cong M_{2^p}(B\odot S^q)\xrightarrow{M_{2^p}(g)}M_{2^p}(M_{2^q}(C))\bl{can}\cong M_{2^{p+q}}(C),$$
where the middle isomorphism uses the canonical isomorphism $M_\ell(S\times T)\cong M_\ell(S)\times M_\ell(T)$.

The maps~\eqref{bim} are such that the ``bilinearity diagram"~\eqref{yugra} is commutative. The pairing~\eqref{avtomobilist}
is constructed. Since the composition of matrix correspondences is strictrly associative, then the
pairings~\eqref{avtomobilist} are strictly associative.

Next, let $\mathbf{1}_0:S^0\to\Mtrc_0(A,A)$ be the map which
sends the basepoint to the zero homomorphism and the non-basepoint to the identity homomorphism
$\id_A\in\Mtr_0(A,A)$. We get a $\Sigma_p$-equivariant map $\mathbf{1}_p:S^p\to\Mtrc_p(A,A)$
defined as the composition
   $$S^p\cong S^0\wedge S^p\xrightarrow{{\mathbf 1}_0\wedge S^p}\Mtrc_0(A,A)\wedge S^p\xrightarrow{\lambda_{0,p}}\Mtrc_p(A,A),$$
where $\lambda_{0,p}$ is a natural map obtained from structure maps of the symmetric spectrum $\Mtrc_*(A,A)$.
So we get a canonical map $\mathbf 1:\bb S\to\Mtrc_*(A,A)$, where $\bb S$ is the sphere spectrum.

We document the above arguments as follows.

\begin{theorem}\label{avangard}
The category $\aha$ can be enriched over symmetric spectra if we take $\Mtrc_*(A,B)$, $A,B\in\aha$, for
symmetric spectra of morphisms and $\mathbf 1:\bb S\to\Mtrc_*(A,A)$ for unit morphisms.
The resulting spectral category is denoted by $\Mtrc_*(k)$. Moreover, the spectral category $\Mtrc_*(k)$ is symmetric monoidal
with respect to the tensor product of algebras.
\end{theorem}

\begin{proof}
It remains to show that $\Mtrc_*(k)$ is symmetric monoidal. 
Let $\Delta^{\op}\Alg_k^{\op}$
be the symmetric monoidal category of simplicial objects in $\Alg_k^{\op}$.
Similarly to the proof of Theorem~\ref{mtrsigma}
$\Mtrc_*(k)$ is dual (up to isomorphism) to a spectral category $\corr_*^{E,\sigma}(\Aff_k)$ of $(E,\sigma)$-correspondences
in the sense of~\cite[Section~5]{GTLMS} whose objects are those of $\Aff_k$,
$P=M_2(k)\in\Alg_k^{\op}$ and $E=(k\odot S^0,k\odot S^1,k\odot S^2,\ldots)$. Then $E$ is a commutative ring object in
the category of symmetric sequences $(\Delta^{\op}\Alg_k^{\op})^\Sigma$. The spectral category
$\corr_*^{E,\sigma}(\Aff_k)$ is symmetric monoidal by~\cite[Theorem~5.2]{GTLMS}, and hence so is $\Mtrc_*(k)$.
\end{proof}

\section{The spectral category $\bb K$}\label{sectionbbk}

In this section we follow~\cite[Section~3]{Gar1}
and~\cite[Section~4]{GP1} to introduce a spectral category $\bb K$.
This spectral category is associated with a 2-category whose objects
are those of $\aha$ and each category of morphisms $\cc A(A,B)$,
$A,B\in\aha$, is additive. The {\it underlying category\/} $\cc U\bb
K$ of $\bb K$ in the sense of~\cite[p.~316]{Bor} is defined by
couples
   $$\Phi=(n,\alpha:A\to M_nB),\quad n\in\bb Z_{n\geq 0},\quad A,B\in\aha,$$
where $\alpha$ is a non-unital algebra homomorphism. By definition, we identify all couples with $\alpha=0$. Precisely,
$(n,0:A\to M_nB)=(\ell,0:A\to M_\ell B)$ for any $n,\ell$. By convention, $0:=(0,0:A\to M_0(B))$ is the class
of all couples with $\alpha=0$. With these definitions each set of morphisms equals
   $$\cc U\bb K(A,B)=\bigvee_{n\geq 0}\{(n,\alpha:A\to M_nB)\mid\textrm{$\alpha$ is non-unital}\},$$
where each summand is pointed at $(n,0:A\to M_nB)$.

A composition law
   \begin{equation}\label{compos}
    \cc U\bb K(A,B)\wedge\cc U\bb K(B,C)\to\cc U\bb K(A,C)
   \end{equation}
in $\cc U\bb K$ is defined by the rule
   $$(\alpha:A\to M_nB,\beta:B\to M_lC)\longmapsto \beta\circ\alpha:=M_n(\beta)\alpha:A\to M_{nl}C.$$
Here $M_n(\beta)$ is the composition of the homomorphism $M_nB\to
M_nM_lC$, induced by $\beta$, and the natural isomorphism $M_nM_lC\cong M_{nl}C$.

Given $A,B\in\aha$ define an additive category $\cc A(A,B)$ as follows. Its objects are pairs
$(n,\alpha:A\to M_nB)\in\Hom_{\cc U\bb K}(A,B)$. Note that $p(\alpha):=\alpha(1)$ is an idempotent matrix of $M_n(B)$
and we write $P(\alpha)$ to denote the $(A,B)$-bimodule $\im(p(\alpha):B^n\to B^n)$. By construction,
$P(\alpha)$ is a finitely generated projective right $B$-module. Given two objects
$\Phi=(n,\alpha:A\to M_nB)$ and $\Phi'=(n,\alpha':A\to M_\ell B)$
   $$\Hom_{\cc A(A,B)}(\Phi,\Phi')=\Hom_{A\otimes B^{\op}}(P(\alpha),P(\alpha')).$$
Here $\Hom_{A\otimes B}$ stands for the Hom-group of homomorphisms of $(A,B)$-bimodules. The composition law
of $(A,B)$-bimodules induces a composition law in $\cc A(A,B)$ making it a
preadditive category. Note that the distinguished zero object is the couple $0:=(0,0:A\to M_0(B))$.
In fact, $\cc A(A,B)$ is additive if set
   $$(n,\alpha:A\to M_nB)\oplus(\ell,\beta:A\to M_\ell B)=(n+\ell,A\xrightarrow{(\alpha,\beta)} M_nB\times M_\ell(B)\hookrightarrow M_{n+\ell}(B)).$$
Observe that $P(\alpha\oplus\beta)\cong P(\alpha)\oplus P(\beta)$. Moreover,
   $$(n,\alpha:A\to M_nB)\oplus 0=0\oplus(n,\alpha:A\to M_nB)=(n,\alpha:A\to M_nB).$$

We can define a bilinear pairing
   \begin{equation}\label{luu}
    \cc A(A,B)\times\cc A(B,C)\lra{\circ}\cc A(A,C),\quad A,B,C\in\aha.
   \end{equation}
It is defined on objects by~\eqref{compos}. In order to define pairing~\eqref{luu} on morphisms, choose $\Phi_1=(n_1,\alpha_1:A\to M_{n_1}B)\in\cc A(A,B)$
and $\Phi_2=(n_2,\alpha_2:B\to M_{n_2}C)\in\cc A(B,C)$. Consider the diagram
   $$\xymatrix{B^{n_1}\otimes_B P(\alpha_2)
               \ar[r]^{can}_{\cong}&P(\alpha_2)^{n_1}\ar@{ >->}[r]^{i_1}&(C^{n_2})^{n_1}
               \ar[r]^{\ell}_\cong&C^{n_2n_1}\ar@<1ex>[d]^{p(\alpha_2\circ\alpha_1)}\\
               P(\alpha_1)\otimes_B P(\alpha_2)\ar@{ >->}[u]^{i_1\otimes\id}\ar[rrr]^{\sigma_{12}}
               &&&P(\alpha_2\circ\alpha_1),\ar@<1ex>[u]^{i(\alpha_2\circ\alpha_1)}}$$
where $\sigma_{12}=p(\alpha_2\circ\alpha_1)\circ\ell\circ i_1\circ
can\circ(i_2\otimes\id)$ (here $\ell (e_{j,i})=e_{i+(j-1)n_1}$). Note that the isomorphism $\ell$
induces an algebra isomorphism $M_{n_2}(M_{n_1}(C))\cong M_{n_2n_1}(C)$ which coincides with the canonical
isomorphism inserting $(n_1,n_1)$-matrices into entries of a $(n_2,n_2)$-matrix.
Let $\beta_1:\Phi_1\to\Phi_1'$ and $\beta_2:\Phi_2\to\Phi_2'$ be
morphisms in $\cc A(A,B)$ and $\cc A(B,C)$ respectively. We set
   \begin{equation*}\label{oleg}
    \beta_2\star\beta_1=\sigma_{12}'\circ(\beta_2\otimes\beta_1)\circ\sigma_{12}^{-1}:P(\alpha_2\circ\alpha_1)
    \to P(\alpha_2'\circ\alpha_1').
   \end{equation*}
The definition of pairing~\eqref{luu} is finished.

The functor $\cc A(A,B)\times\cc A(B,C)\lra{\circ}\cc A(A,C)$ is plainly bilinear for all $A,B,C\in\aha$.
Moreover,
   $$(\Phi_1\oplus\Phi_2)\circ\Phi_3=(\Phi_1\circ\Phi_3)\oplus(\Phi_2\circ\Phi_3),\quad
     \Phi_1\circ(\Phi_3\oplus\Phi_4)=(\Phi_1\circ\Phi_3)\oplus(\Phi_1\circ\Phi_4)$$
for all $\Phi_1,\Phi_2\in\cc A(A,B)$ and $\Phi_3,\Phi_4\in\cc A(B,C)$. Furthermore,
   $$\beta_3\star(\beta_2\star\beta_1)=(\beta_3\star\beta_2)\star\beta_1$$
for any $\beta_1\in\Mor\cc A(A,B)$, $\beta_2\in\Mor\cc A(B,C)$, $\beta_3\in\Mor\cc A(C,D)$.

If we define an object $\id_A\in\Ob\cc A(A,A)$, $A\in\aha$, by $\id_A=(1,\id:A\to A)$ then
for the functors $\{\id_A\}\times\cc A(A,B)\lra{\circ}\cc A(A,B)$ and $\cc
A(A,B)\times\{\id_B\}\lra{\circ}\cc A(A,B)$ are identities on $\cc
A(A,B)$.

For any $A,B,C,D\in\aha$ the diagram of functors\footnotesize
   \begin{equation}\label{wowwow}
    \xymatrix{&(\cc A(A,B)\times\cc A(B,C))\times\cc A(C,D)\ar[r]^(.585){\circ\times\id}&\cc A(A,C)\times\cc A(C,D)\ar[dd]_{\circ}\\
              \cc A(A,B)\times(\cc A(B,C)\times\cc A(C,D))\ar[ur]^{\cong}\ar[dr]_{\id\times\circ}\\
              &\cc A(A,B)\times\cc A(B,D)\ar[r]^{\circ}&\cc A(A,D)}
   \end{equation}
\normalsize is strictly commutative by construction.

Suppose $\cc M$ is an additive category. We follow the same constructions as in~\cite[Section~8.7]{Rog} to
define Waldhausen's symmetric spectrum $K(\cc M)$~\cite{Wal}. Given a positive integer $n$, one can define the $n$-fold multisimplicial
additive category $S^{\oplus,n}\cc M:=S^{\oplus}\bl n\ldots S^{\oplus}\cc M$. The $n$th space of Waldhausen's $K$-theory spectrum
is given by\label{KGr}
   $$K(\cc M)_n=|\Ob(S^{\oplus,n}\cc M)|$$
with the right hand side the diagonal of the $n$-fold multisimplicial set $\Ob (S^{\oplus,n}\cc M)$. The $n$th symmetric
group $\Sigma_n$ acts on $K(\cc M)_n$ by permuting the order of the $S^{\oplus}$-constructions. Each structure map $\sigma$ is the composite
    $$|\Ob(S^{\oplus,n}\cc M)|\wedge S^1\cong|\Ob(S^{\oplus,n}S^{\oplus}\cc M)|^{(1)}
    \subset|\Ob(S^{\oplus,n}S^{\oplus}\cc M)|\cong|\Ob(S^{\oplus,n+1}\cc M)|,$$
where the superscript ${}^{(1)}$ stands for the 1-skeleton with
respect to the last simplicial direction. The $k$-fold iterated
structure map $\sigma^k$ is then defined as the composite
   \begin{gather*}
    |\Ob(S^{\oplus,n}\cc M)|\wedge S^k\cong|\Ob(S^{\oplus,n}S^{\oplus}\bl k\ldots S^{\oplus}\cc M)|^{(1,\ldots,1)}\subset\\
    \subset|\Ob(S^{\oplus,n}S^{\oplus}\bl k\ldots S^{\oplus}\cc M)|\cong|\Ob(S^{\oplus,n+k}\cc M)|,
  \end{gather*}
where the superscript ${}^{(1,...,1)}$ indicates the
multi-1-skeleton with respect to the $k$ last simplicial directions.
This map is plainly $(\Sigma_n\times\Sigma_k)$-equivariant. With
these definitions $K(\cc M)$ becomes a (semistable) symmetric spectrum.

Given $A,B\in\aha$, one sets
   $$\bb K(A,B):=K(\cc A(A,B)).$$
Pairing~\eqref{luu} yields a pairing of symmetric spectra
   \begin{equation}\label{forel}
    \bb K(A,B)\wedge\bb K(B,C)\to\bb K(A,C).
   \end{equation}
Commutativity of the diagram~\eqref{wowwow} implies that~\eqref{forel} is a strictly
associative pairing. Moreover, for any $A\in\aha$ there is a map
$\mathbf 1:S\to\bb K(A,A)$ which is subject to the unit coherence
law. Note that $\mathbf 1_0:S^0\to\bb K(A,A)_0$ is the map which sends the basepoint to the null object
and the non-basepoint to the unit object $\id_A$.

Denote by $\cc P(A,B)$ the additive category of those $(A,B)$-bimodules
which are finitely generated projective right $B$-modules. Its Waldhausen $K$-theory 
symmetric spectrum will be denoted by $K^\oplus(\cc P(A,B))$. Similarly to~\cite[Definition~5.7]{GP1}
there is a natural additive functor between additive categories
   $$a_{A,B}:\cc A(A,B)\to\cc P(A,B).$$
This functor is plainly an equivalence of categories.

We document the above constructions as follows.

\begin{theorem}[\cite{GP1}]\label{semga}
The triple $(\bb K,\wedge,\mathbf 1)$ determines a spectral category
on $\aha$. For every $A,B\in\aha$ the equivalence of additive categories
$\cc A(A,B)\to\cc P(A,B)$ induces a stable equivalence of symmetric $K$-theory spectra
$\bb K(A,B)\lra{\sim}K^\oplus(\cc P(A,B))$.
\end{theorem}

\section{Additivity Theorem}\label{sectionadd}

For every $A,B,C\in\aha$ and $n\geq 0$ the set $\Mtr_n(A\times
B,C)$ is isomorphic to the set $\Mtr_n'(A\times B,C)$ consisting of
couples $(f,g)\in\Mtr_n(A,C)\times\Mtr_n(B,C)$ such that
$f(a)\cdot g(b)=g(b)\cdot f(a)=0$ for any $a\in A,b\in B$. Indeed,
to every $h\in\Mtr_n(A\times B,C)$ we associate the couple $(h\circ
i_A,h\circ i_B)$ with $i_A:A\to A\times B$, $i_A(a)=(a,0)$, and
$i_B:B\to A\times B$, $i_B(b)=(0,b)$. On the other hand, to every
$(f,g)\in\Mtr_n'(A\times B,C)$ one associates a matrix
correspondence $f\uplus g$ of level $n$ defined as $(f\uplus
g)(a,b)=f(a)+g(b)$. These two associations give the isomorphism in
question.

For every $n\geq 0$ there is a natural map of simplicial sets
   $$\delta_n:\Mtr_n(A,C^\Delta)\times\Mtr_n(B,C^\Delta)\to\Mtr_{n+1}(A\times B,C^\Delta),\quad(f,g)\mapsto((a,b)\mapsto\diag(f(a),g(b))),$$
as well as a map
   $$\vartheta_n=(i_A^*,i_B^*):\Mtr_n(A\times B,C^\Delta)\to\Mtr_n(A,C^\Delta)\times\Mtr_n(B,C^\Delta).$$

\begin{lemma}\label{neftehimik}
For every $n\geq 0$ there is a natural map of simplicial sets
   $$H_n:\Mtr_{n}(A\times B,C^\Delta)\to\Mtr_{n+1}(A\times B,C^\Delta[x]),$$
functorial in matrix correspondences of level zero in both
arguments, such that $\partial^0_xH_n=\sigma_{A\times B}$ and
$\partial^1_xH_n=\delta_n\vartheta_n$.
\end{lemma}

\begin{proof}
It is enough to show the assertion for $\Mtr_{n}'(A\times B,C^\Delta)$.
Consider the matrix~\eqref{rot} $T\in SL_2(M_{2^{n}}(k[x]))$. For
every $b\in B$ and $g\in\Mtr_{n}(B,C^{\Delta^k})$ we have
   $$T\cdot\left(\begin{array}{cc}g(b)&0\\ 0 & 0 \\ \end{array}\right)\cdot T^{-1}=
     g(b)\cdot\left(\begin{array}{cc}(1-x^2)^2&(1-x^2)(2x-x^3)\\ x(1-x^2) & x(2x-x^3) \\ \end{array}\right).$$
It follows that for every $(f,g)\in\Mtr_{n}'(A\times B,C^{\Delta^k})$ the
couple $(f\sigma_A,T(g\sigma_B)T^{-1})$ belongs to $\Mtr_{n}'(A\times B,C^{\Delta^k}[x])$. The desired map
   $$H_n:\Mtr_{n}'(A\times B,C^\Delta)\to\Mtr_{n+1}'(A\times B,C^\Delta[x])$$
is then defined in level $k$ by $H_{n,k}(f,g)=(f\sigma_A,T(g\sigma_B)T^{-1})$. 
We tacitly use here~\cite[Proposition~3.2]{Gar}. One has
that $\partial^0_xH_{n,k}=(f\sigma_{A},g\sigma_B)$ and
$\partial^1_xH_{n,k}=(f\sigma_{A},g\sigma_B')$ what precisely corresponds
to $\sigma_{A\times B}$ and $\delta_n\vartheta_n$ under the
isomorphism of simplicial sets $\Mtr_{n+1}(A\times B,C^\Delta)\cong\Mtr_{n+1}'(A\times
B,C^\Delta)$.
\end{proof}

\begin{corollary}\label{torpedo}
Every $h\in\pi_{\ell}(\Mtr(A\times B,C^\Delta))$ equals
$(i_A\circ pr_A+ i_B\circ pr_B)(h)$, where $p_A:A\times B\to
A$, $p_B:A\times B\to B$ natural projections and the sum is
inherited from the $H$-space structure of Theorem~\ref{slovan}.
\end{corollary}

\begin{lemma}\label{obv}
Suppose $X_n\subseteq X_{n+1}$ is a directed system of inclusions of
simplicial sets, $Y_n\subseteq  Y_{n+1}$ is a directed system of simplicial subsets
$Y_n\subseteq X_n$, and $p_n\colon X_n\to Y_{n+1}$ is a sequence of
maps such that the restriction of $p_n$ to $Y_n$ equals the inclusion
$Y_n\subseteq Y_{n+1}.$ Assume that for every $n$ there is a homotopy
$H(n)\colon X_n\times \Delta[1]\to X_{n+1}$ such that $H(n)_1\colon X_n\to X_{n+1}$
is the inclusion map, $H(n)(Y_n\times\Delta[1])\subseteq Y_{n+1}$, and
the map $H(n)_0\colon X_n\to X_{n+1}$ equals the composition
$X_n\xrightarrow{p_n} Y_{n+1}\subseteq X_{n+1}.$
 Then the inclusion $Y\to X$ is a weak equivalence, where $Y=\colim_n Y_n, X=\colim_n X_n$.
\end{lemma}

\begin{proof}
The proof is like that of~\cite[Lemma~12.14]{GN}.
Consider a point $y\in Y_n$. The inclusion map $j\colon X_n\to X_{n+1}$ and the
composition $f\colon X_n\xrightarrow{p_n} Y_{n+1}\subseteq X_{n+1}$ are homotopic
by means of the free homotopy $H(n)$. Then the two induced maps
\[\pi_i(j),\pi_i(f)\colon\pi_i(X_n,y)\to\pi_i(X_{n+1},y)\]
differ by the action $[\gamma]_*$ of the class $[\gamma]\in\pi_1(X_{n+1},y)$ on $\pi_i(X_{n+1},y)$,
where $\gamma\colon\Delta[1]\to Y_{n+1},\gamma(t)=H(n)(y,t)$, is the
loop given by the image of the base point $y$ under the homotopy $H(n)$. Since the
loop $\gamma$ lies inside $Y_{n+1}$, the action of $[\gamma]$ on $\pi_i(Y_{n+1},y)$
preserves the image of $\pi_i(Y_{n+1},y)$ under the inclusion map $Y_{n+1}\to X_{n+1}$. Then the image
   \[\pi_i(j)(\pi_i(X_n,y))=[\gamma]_*\pi_i(f)(\pi_i(X_n,y))\]
lies inside the image of $\pi_i(Y_{n+1},y)$, and therefore
$\pi_i(Y,y)\to\pi_i(X,y)$ is surjective for any point $y\in Y$. The
existence of maps $p_n$ implies that the map
$\pi_i(Y,y)\to\pi_i(X,y)$ is also injective. Furthermore, for every point
$x\in X_n$ the map $t\mapsto H(n)(x,t)$ gives a path between $j(x)$ and
a point of $Y_{n+1}$. We see that $\pi_0(Y)\to\pi_0(X)$ is surjective,
and hence $Y\to X$ is a weak equivalence.
\end{proof}

The following result is reminiscent of the Additivity Theorem 
of~\cite{GP3} for framed motivic spaces $C_*\Fr(-,X)$
associated to smooth algebraic varieties.

\begin{theorem}[Additivity]\label{add}
For every $A,B,C\in\aha$ the map of simplicial sets
   $$\vartheta=(i_A^*,i_B^*):\Mtr(A\times B,C^\Delta)\to\Mtr(A,C^\Delta)\times\Mtr(B,C^\Delta),$$
induced by the non-unital homomorphisms $i_A:A\to A\times B$,
$i_A(a)=(a,0)$, and $i_B:B\to A\times B$, $i_B(b)=(0,b)$, is a weak
equivalence.
\end{theorem}

\begin{proof}
As above we have a map
   $$\delta_{n}:\Mtr_n(A,C^\Delta)\times\Mtr_n(B,C^\Delta)\to\Mtr_{n+1}(A\times B,C^\Delta),\quad n\geq 0.$$
Consider a diagram in $\bb S_\bullet$, in which horizontal and vertical maps are monomorphisms:
   $$\xymatrix{\cdots\ar[r]^(.3){\sigma_{A\times B}}&\Mtr_{n}(A\times B,C^\Delta)\ar[r]^{\sigma_{A\times B}}\ar[d]_{\vartheta_n}
                       &\Mtr_{n+1}(A\times B,C^\Delta)\ar[r]^(.7){\sigma_{A\times B}}\ar[d]^{\vartheta_{n+1}}&\cdots\\
                       \cdots\ar[r]_(.23){\sigma_A\times\sigma_B}&\Mtr_n(A,C^{\Delta})\times\Mtr_n(B,C^{\Delta})\ar[r]_(.47){\sigma_A\times\sigma_B}\ar[ur]^{\delta_n}
                       &\Mtr_{n+1}(A,C^{\Delta})\times\Mtr_{n+1}(B,C^{\Delta})\ar[r]_(.8){\sigma_A\times\sigma_B}&\cdots}$$
We have that $(\sigma_A\times\sigma_B)\circ\vartheta_n=\vartheta_{n+1}\circ\sigma_{A\times B}$ for all $n$.
It follows from Lemma~\ref{neftehimik} and~\cite[Proposition~3.2]{Gar} that 
$\delta_n\circ\vartheta_n$ is homotopic to $\sigma_{A\times B}$ by means of the homotopy
      $$H_n:\Mtr_{n}(A\times B,C^\Delta)\to\Mtr_{n+1}(A\times B,C^\Delta[x]).$$

In turn, the map $\vartheta_{n+1}\circ\delta_n$ takes $(f,g)$ to $(f\sigma_A,g\sigma'_B)$. The map
   $$G_n:\Mtr_{n}(A,C^{\Delta})\times\Mtr_{n}(B,C^{\Delta})\to\Mtr_{n+1}(A,C^{\Delta}[x])\times\Mtr_{n+1}(B,C^{\Delta}[x])$$
taking $(f,g)$ to $(f\sigma_A,T\circ g\sigma_B\circ T^{-1})$, where $T$ is the rotational matrix~\eqref{rot}, yields a simplicial homotopy between
$\sigma_A\times\sigma_B$ and $\vartheta_{n+1}\circ\delta_n$. Moreover,
   $$G_n\circ\vartheta_n=\vartheta_{n+1}[x]\circ H_n.$$
   
By~\cite[Proposition~2.3]{CS} there is a homotopy
   $$\tilde H(n):Ex^\infty(\Mtr_n(A,C^{\Delta}))\times Ex^\infty(\Mtr_n(B,C^{\Delta}))\times\Delta[1]\to Ex^\infty(\Mtr_{n+1}(A\times B,C^\Delta))$$
such that $\tilde H(n)_0\circ\vartheta_n=\sigma_{A\times B}$, $\tilde H(n)_1=\delta_n$ and $\tilde H(n)\circ\vartheta_n=Ex^\infty(H_n)$.
Next, the homotopy $G_n$ yields a homotopy
   \begin{multline*}
    G(n):Ex^\infty(\Mtr_{n}(A,C^{\Delta}))\times Ex^\infty(\Mtr_{n}(B,C^{\Delta}))\times\Delta[1]\to\\
       \to Ex^\infty(\Mtr_{n+1}(A,C^{\Delta}))\times Ex^\infty(\Mtr_{n+1}(B,C^{\Delta}))
   \end{multline*}
such that $G(n)_1=\vartheta_{n+1}\delta_n$ and $G(n)_0=\sigma_A\times\sigma_B$.
Moreover, the restriction of the map $G(n)$ to $Ex^\infty(\Mtr_{n}(A\times B,C^{\Delta}))\times\Delta[1]$
factors through $Ex^\infty(\Mtr_{n+1}(A\times B,C^\Delta))$ and equals $\vartheta_{n+1}\circ Ex^\infty(H_n)$.

The homotopies $\vartheta_{n+1}\circ\tilde H(n),G(n)$ yield a map of simplicial sets
   \begin{multline*}
    I(n):Ex^\infty(\Mtr_{n}(A,C^{\Delta}))\times Ex^\infty(\Mtr_{n}(B,C^{\Delta}))\times\sd^1\Delta[1]\to\\
       \to Ex^\infty(\Mtr_{n+1}(A,C^{\Delta}))\times Ex^\infty(\Mtr_{n+1}(B,C^{\Delta}))
   \end{multline*}
such that $I(n)_0=\vartheta_{n+1}\circ\tilde H(n)_0$ and $I(n)_1=\sigma_A\times\sigma_B$.
The proof of the theorem now follows from Lemma~\ref{obv} if we take $p_n=\tilde H(n)_0$ with homotopy $I(n)$ in it.
\end{proof}

\begin{definition}\label{mtrmot}
The {\it matrix motive $M_{\mathrm{mtr}}(A)$ of a $k$-algebra $A\in\aha$} is the spectral functor taking
$B\in\aha$ to the symmetric $S^1$-spectrum 
   $$M_{\mathrm{mtr}}(A)(B)=(\Mtr(A,B^\Delta),\Mtr(A\odot S^1,B^\Delta),\Mtr(A\odot S^2,B^\Delta),\ldots).$$
It is Segal's symmetric $S^1$-spectrum associated to the $\Gamma$-space
   $$K\in\Gamma^{\op}\mapsto\Mtr(A\odot K,B^\Delta)\in\bb S_\bullet$$
Note that the definition of the matrix motive of an algebra is similar to that for the framed motive of a smooth
algebraic variety in the sense of~\cite{GP3}.
\end{definition}

\begin{corollary}\label{gammacor}
For every $A,B\in\aha$ the $\Gamma$-space
   $$K\in\Gamma^{\op}\mapsto\Mtr(A\odot K,B^\Delta)\in\bb S_\bullet$$
is special. In particular, the symmetric $S^1$-spectrum
   $$M_{\mathrm{mtr}}(A)(B)=(\Mtr(A,B^\Delta),\Mtr(A\odot S^1,B^\Delta),\Mtr(A\odot S^2,B^\Delta),\ldots)$$ 
is an $\Omega$-spectrum in positive degrees.
\end{corollary}

\section{Comparison Theorem}\label{sectioncompar}

\begin{proposition}\label{segalcat}
For every $A,B,C\in\aha$ the additive functor of additive categories
   $$\theta=(i_A^*,i_B^*):\cc A(A\times B,C)\to\cc A(A,C)\times\cc A(B,C),$$
induced by the non-unital homomorphisms $i_A:A\to A\times B$,
$i_A(a)=(a,0)$, and $i_B:B\to A\times B$, $i_B(b)=(0,b)$, is an equivalence of categories.
\end{proposition}

\begin{proof}
Let $\Phi=(\ell,\gamma:A\times B\to M_\ell(C))$ be an object of $\cc A(A\times B,C)$. Then $\gamma$
equals the composite
   $$A\times B\lra{u}\End_CP(\gamma)\hookrightarrow M_\ell(C),$$
where $u$ is a unital homomorphism of $k$-algebras. Let
$\gamma_A=(A\times B\bl{pr_A}\twoheadrightarrow A\bl{i_A}\hookrightarrow A\times B\xrightarrow{\gamma} M_\ell(C))$ and
$\gamma_B=(A\times B\bl{pr_B}\twoheadrightarrow B\bl{i_B}\hookrightarrow A\times B\xrightarrow{\gamma} M_\ell(C))$.
The $(A\times B,C)$-bimodule
$P(\gamma)$ equals $P(\gamma_A)\oplus P(\gamma_B)$ and $\theta(P(\gamma))=(P(\gamma_A),P(\gamma_B))$.

If $\Phi'=(m,\gamma':A\times B\to M_m(C))$ is another object of $\cc A(A\times B,C)$, then every $(A\times B,C)$-bimodule
homomorphism $g:P(\gamma)\to P(\gamma')$ is of the form $g_A\oplus g_B:P(\gamma_A)\oplus P(\gamma_B)\to P(\gamma_A')\oplus P(\gamma_B')$
due to the facts that $i_Apr_A,i_Bpr_B$ are orthogonal idempotents of $A\times B$, $1_{A\times B}=i_Apr_A+i_Bpr_B$
and $g\circ u(i_Apr_A)=u'(i_Apr_A)\circ g$, $g\circ u(i_Bpr_B)=u'(i_Bpr_B)\circ g$.
Here $g_A=g\circ u(i_Apr_A)|_{P(\gamma_A)}:P(\gamma_A)\to P(\gamma_A')$ and $g_B=g\circ u(i_Apr_B)|_{P(\gamma_B)}:P(\gamma_B)\to P(\gamma_B')$.
It follows that $\theta=(i_A^*,i_B^*):\cc A(A\times B,C)\to\cc A(A,C)\times\cc A(B,C)$ is fully faithful.

A quasi-inverse functor
   $$\tau:\cc A(A,C)\times\cc A(B,C)\to\cc A(A\times B,C)$$
to $\theta$ is defined as follows. Given $\Phi=(n,\alpha:A\to M_n(C))\in\cc A(A,C)$ and
$\Psi=(k,\beta:B\to M_k(C))\in\cc A(B,C)$,
the functor $\tau$ sends $(\Phi,\Psi)$ to $(A\times B\xrightarrow{({\alpha},{\beta})}M_{n}(C)\times M_{k}(C)\hookrightarrow M_{n+k}(C))$.
It is straightforward to extend $\tau$ to morphisms. Then there is a natural isomorphism $\theta\tau(\Phi,\Psi)\cong(\Phi,\Psi)$
constructed by means of conjugations with respect to permutations matrices.
We see that $\theta$ is also essentially surjective, and hence an equivalence of categories.
\end{proof}

\begin{corollary}\label{segalcatcor}
For every $A,B\in\aha$ the assignments
   $$K\in\Gamma^{\op}\mapsto\cc A(A^K,B)$$
determine a $\Gamma$-category in the sense of Segal~\cite[Definition~2.1]{Seg} such that $A(A^\emptyset,B)$ is a 
trivial category with one object and one morphism. In particular, the assignments
   $$K\in\Gamma^{\op}\mapsto i\cc A(A^K,B)$$
determine a special $\Gamma$-category, where $i\cc A(A^K,B)$ is the subcategory of isomorphisms of $\cc A(A^K,B)$.
\end{corollary}

The Segal machine~\cite{Seg} yields a symmetric $S^1$-spectrum
   $$K^{Seg}(A,B)=(|i\cc A(A,B)|, |i\cc A(A\odot S^1,B)|, |i\cc A(A\odot S^2,B)|,\ldots ).$$
Clearly, it is functorial (in both arguments) in morphisms of the category $\cc U\bb K$, i.e. functorial in non-unital ring
homomorphisms of the form $f:C\to M_k(D)$. 

Denote by $\Mtr_i^{Seg}(A,B)$, $i\geq 0$, the symmetric $S^1$-spectrum associated to the Segal space
   $$K\in\Gamma^{\op}\mapsto\Mtr_i(A\odot K,B).$$

Recall that the underlying category $\cc U\bb K$ of the spectral category $\bb K$ consists of the non-unital ring
homomorphisms $f:C\to M_k(D)$, $C,D\in\aha$. Thus,
   $$\Hom_{\Mod\bb K}(\bb K(A,-),\bb K(B,-))=\Hom_{\cc U\bb K}(B,A).$$
We can therefore construct a commutative diagram of symmetric spectra
   $$\xymatrix{\Mtr^{Seg}_0(A,-)\ar[d]_{\sigma_A}\ar[r]^{\eta}&K^{Seg}(A,-)\ar[r]^\kappa\ar[d]_{\sigma_A}&\bb K(A,-)\ar[d]^{\sigma_A}\\
               \Mtr^{Seg}_1(A,-)\ar[d]_{\sigma_A}\ar[r]^{\eta}&K^{Seg}(A,-)\ar[d]_{\sigma_A}\ar[r]^\kappa &\bb K(A,-)\ar[d]^{\sigma_A}\\
               \Mtr^{Seg}_2(A,-)\ar[d]_{\sigma_A}\ar[r]^{\eta}&K^{Seg}(A,-)\ar[d]_{\sigma_A}\ar[r]^\kappa &\bb K(A,-)\ar[d]^{\sigma_A}\\
               \vdots&\vdots&\vdots}$$
Here $\eta$ is the obvious map sending each space $\Mtr_i^{Seg}(A\odot S^m,-)$ to itself if we regard it as a simplicial space 
trivial in the $i.$-direction. As $\sigma_A:\cc A(A,B)\to\cc A(A,B)$ is an equivalence of categories, the middle and right vertical arrows
of the diagram are objectwise stable equivalences. We shall write $\bb K^\sigma(A,-)$ for a colimit of the right vertical arrows. Then
the natural map
   $$\sigma_A:\bb K(A,-)\to\bb K^\sigma(A,-)$$ 
is an objectwise stable equivalence of spectra.

To describe $\kappa$, we need some preparations. 
Following Corollary~\ref{gammacor} we write $\Mtr^{Seg}(A,-)$ for a colimit of the left vertical 
arrows of the diagram. It yields a functor
   $$B\in\aha\mapsto\Mtr^{Seg}(A,B)\in Sp_{S^1}^\Sigma.$$
Note that $\Mtr^{Seg}(A,B^\Delta)=M_{\mathrm{mtr}}(A)(B)$.

First, define a simplicial additive functor $\kappa:\cc A(A\odot S^1,B)\to S_\bullet\cc A(A,B)$ between simplicial additive categories,
where $S_\bullet$ refers to the Waldhausen $S_\bullet$-con\-struc\-tion~\cite{Wal}. On the level of 0-simplicies $\kappa$
takes the zeroth homomorphism, i.e. a unique 0-simplex of $\cc A(A\odot S^1,B)$, to itself, which is the only 0-simplex of $S_\bullet\cc A(A,B)$.
Recall that the set $S^1_n$ of $n$-simplices of $S^1$ equals $[n]=\{0,1,\ldots,n\}$, pointed at 0. Given an $n$-simplex
$f:\times_1^n A\to M_l(B)$, we have to associate an $n$-simplex $\kappa(f):\Ar[n]\to\cc A(A,B)$ to it. Here $\Ar[n]$ stands for
the category of arrows of the ordered set $0<1<\cdots<n$ whose objects are the couples $(i,j)$ with $i\leq j$.

For every $i\leq n$ let $\diag_i:A\to\times_1^n A$ denote the non-unital homomorphism sending $a\in A$ to the $n$-tuple
$(a,a,\ldots,a,0,\ldots,0)$ with the last $n-i$ entries zero. 
Given $i\leq j\leq n$ one sets,
   $$\kappa(f)(i,j):=(A\xrightarrow{\diag_j-\diag_i}\times_1^n A\xrightarrow f M_l(B)).$$
By construction, $\kappa(f)(i,k)=\kappa(f)(i,j)+\kappa(f)(j,k)$ for every $i\leq j\leq k\leq n$ giving
canonical isomorphisms $P(\kappa(f)(i,k))=P(\kappa(f)(i,j))\oplus P(\kappa(f)(j,k))$ and canonical split short exact sequence of $(A,B)$-bimodules
   $$P(\kappa(f)(i,j))\hookrightarrow P(\kappa(f)(i,k))\twoheadrightarrow P(\kappa(f)(j,k)).$$
These short exact sequences determine the desired $n$-simplex $\kappa(f):\Ar[n]\to\cc A(A,B)$. The construction of
$\kappa(f)$ is easily extended to the desired simplicial additive functor $\kappa:\cc A(A\odot S^1,B)\to S_\bullet\cc A(A,B)$.

\begin{theorem}[Comparison]\label{compar}
For every $A,B\in\aha$ the zigzag map of symmetric spectra
   $$\kappa_{A,B}\circ\eta:M_{\mathrm{mtr}}(A)(B)\to\bb K^\sigma(A,B^\Delta)\xleftarrow{\sim}\bb K(A,B^\Delta):\sigma_A$$
is a zigzag of stable equivalences, functorial in $A$ and $B$.
\end{theorem}

\begin{proof}
We have already noticed above that $\sigma_A$ is a stable equivalence of spectra. We also tacitly use the fact that geometric realization
preserves stable equivalences of spectra.
We claim that $K^{Seg}(A,B^\Delta)\xrightarrow{\kappa}\bb K(A,B^\Delta)$ is a stable equivalence of spectra.
By Corollary~\ref{segalcatcor} the $\Gamma$-space
   $$S\in\Gamma\mapsto|i\cc A(A^S,B)|$$
is special. Thus $K^{Seg}(A,B)$ is an $\Omega$-spectrum in positive degrees. Therefore our claim reduces to showing that
the map of simplicial sets
   $$\kappa_1:|i\cc A(A\odot S^1,B^\Delta)|\to|iS^\oplus\cc A(A,B^\Delta)|$$
is a weak equivalence.

Consider a commutative diagram of simplicial sets
   $$\xymatrix{|i\cc A(A\times\bl{\ell}\cdots\times A,B^\Delta)|\ar[rr]^{\kappa_1}\ar[dr]_\alpha&&|iS^\oplus_\ell\cc A(A,B^\Delta)|\ar[dl]^\beta\\
                       &|i\cc A(A,B^\Delta)\times\bl{\ell}\cdots\times i\cc A(A,B^\Delta)|}$$
The proof of~\cite[Theorem~10.5]{Gr} implies $\beta$ is a weak equivalence. By Corollary~\ref{segalcatcor}
$\alpha$ is a weak equivalence, and hence so is $\kappa_1$ as claimed.

Next, by Corollary~\ref{gammacor} the $\Gamma$-space
   $$K\in\Gamma^{\op}\mapsto\Mtr(A\odot K,B^\Delta)\in\bb S_\bullet$$
is special. Therefore the fact that $\eta$ is a stable equivalence of spectra 
reduces to showing that the map of simplicial sets
   $$\eta_1:\Mtr(A\odot S^1,B^\Delta)\to|i\cc A^\sigma(A\odot S^1,B^\Delta)|$$
is a weak equivalence, where 
   $$|i\cc A^\sigma(A\odot S^1,B^\Delta)|=\colim(|i\cc A(A\odot S^1,B^\Delta)|\lra\sigma|i\cc A^\sigma(A\odot S^1,B^\Delta)|\lra\sigma\cdots).$$
Recall that $|i\cc A(A,B^\Delta)|$ is the diagonal of the 
bisimplicial set $(m,n)\mapsto i_n\cc A(A,B^{\Delta^m})$. 

By Lemma~\ref{ka}
   $$\Mtr(A,B^\Delta)=\Hom_{\ahaw}(A,M_\infty (B^{\Delta}))=\colim_n\Hom_{\ahaw}(A,M_n(B^{\Delta})).$$
Consider the following map of simplicial sets:
      $$\eta_{0,n}:\Mtr(A,B^\Delta)\to i_n\cc A^\sigma(A,B^\Delta),$$
 where $i_n$ refers to strings of $n$ composable isomorphisms.
Notice that $\eta_{0,0}$ is an isomorphism of simplicial sets. So we assume $n>0$.
Consider a retraction
   $$p_n:i_n\cc A^\sigma(A,B^\Delta)\to\Mtr(A,B^\Delta)$$
to $\eta_{0,n}$ sending a $k$-simplex string of isomorphisms
   $$(\Phi_0\lra{f_0}\Phi_1\lra{f_1}\cdots\xrightarrow{f_{n-1}}\Phi_n)\in i_n\cc A(A,B^{\Delta^k}),\quad\Phi_i=(m_i,\alpha_i:A\to M_{m_i}(B^{\Delta^k})),$$
to $(\alpha_0:A\to M_{m_0}(B^{\Delta^k}))\in (A,M_{m_0}(B^{\Delta^k}))$. Recall that each isomorphism $f_i$ is given by
an isomorphism, denoted by the same letter, of $(A,B)$-bimodules 
$f_i:P(\alpha_i)\to P(\alpha_{i+1})$ with $P(\alpha_i)=\im(\alpha_i:(B^{\Delta^k})^{m_i}\to (B^{\Delta^k})^{m_i})$.

It is convenient to have an equivalent description of the simplicial set $i_n\cc A(A,B^\Delta)$.
Its $k$-simplicies are the ordered tuples 
   $$(g_0:\Phi_0\lra{\cong}\Phi_1,\ldots,g_{n-1}:\Phi_0\lra{\cong}\Phi_n)$$
of isomorphisms in $\cc A(A,B^{\Delta^k})$ compatible with simplicial structure maps in the $\Delta^\bullet$-direction.

For every $0\leq i\leq n$ consider the right $B^{\Delta^k}$-module $Q(\alpha_i)=(I-\alpha_i(1))(B^{\Delta^k})^{m_i}$ associated to the
idempotent matrix $(I-\alpha_i(1))\in M_{m_i}(B^{\Delta^k})$. Then $(B^{\Delta^k})^{m_i}=P(\alpha_i)\oplus Q(\alpha_i)$ and
   $$M_{m_i}(B)=\left(\begin{array}{cc}\End_B P(\alpha_i)&(P(\alpha_i),Q(\alpha_i))\\ (Q(\alpha_i),P(\alpha_i)) & \End_B Q(\alpha_i) \\ \end{array}\right)$$
Also, $B^{m_i+m_0}=P(\alpha_i)\oplus Q(\alpha_i)\oplus P(\alpha_0)\oplus Q(\alpha_0)$ and
\footnotesize
   \begin{equation}\label{matritsa}
    M_{m_i+m_0}(B^{\Delta^k})=\left(\begin{array}{cccc}
       \End_B P(\alpha_i)&(P(\alpha_i),Q(\alpha_i))&(P(\alpha_i),P(\alpha_0))&(P(\alpha_i),Q(\alpha_0))\\ 
       (Q(\alpha_i),P(\alpha_i)) & \End_B Q(\alpha_i)&(Q(\alpha_i),P(\alpha_0))&(Q(\alpha_i),Q(\alpha_0))\\
       (P(\alpha_0),P(\alpha_i))&(P(\alpha_0),Q(\alpha_i))&\End_B P(\alpha_0)&(P(\alpha_0),Q(\alpha_0))\\ 
       (Q(\alpha_0),P(\alpha_i)) &(Q(\alpha_0),Q(\alpha_i))&(Q(\alpha_0),P(\alpha_0))&\End_B Q(\alpha_0)
       \end{array}\right)
   \end{equation}
\normalsize
Note that the composite homomorphism
   $$A\xrightarrow{u(P_i)}\End_B(P_i)\xrightarrow{\tilde\sigma}M_{m_i+m_0}(B^{\Delta^k}),$$
where $u(P_i)$ is the structure homomorphism and $\tilde\sigma$ maps $\End_B(P_i)$ to the (1,1)-entry of the
matrix~\eqref{matritsa}, equals the composite homomorphism
   $$A\xrightarrow{u(P_0)}\End_B(P_0)\xrightarrow{R_i\tilde\sigma'R^{-1}_i}M_{m_i+m_0}(B^{\Delta^k}),$$
where $\tilde\sigma'$ maps $\End_B(P_0)$ to the (3,3)-entry of the matrix~\eqref{matritsa} and
   $$R_i=\left(\begin{array}{cccc}
       0&0&-g_{i-1}^{-1}&0\\ 
       0& 1&0&0\\
       g_{i-1}&0&0&0\\ 
       0&0&0&1
       \end{array}\right),\quad 
       R^{-1}_i=\left(\begin{array}{cccc}
       0&0&g_{i-1}^{-1}&0\\ 
       0& 1&0&0\\
       -g_{i-1}&0&0&0\\ 
       0&0&0&1
       \end{array}\right)$$
Here multiplication of the relevant entries corresponds to the composition
of homomorphisms. By definition,
   $$R_0=\left(\begin{array}{cccc}
       0&0&-1&0\\ 
       0& 1&0&0\\
       1&0&0&0\\ 
       0&0&0&1
       \end{array}\right),\quad 
       R^{-1}_0=\left(\begin{array}{cccc}
       0&0&1&0\\ 
       0& 1&0&0\\
       -1&0&0&0\\ 
       0&0&0&1
       \end{array}\right)$$

We set $R_0(x)=R_0$ and
   $$R_i(x)=\left(\begin{array}{cccc}          1-x^2 & 0 & g_{i-1}^{-1}(x^3-2x) & 0 \\
                                          0 & 1& 0 & 0 \\
                                          g_{i-1}x & 0& 1-x^2 & 0 \\
                                          0 & 0 & 0 & 1 \\
                                        \end{array}\right),\quad
     R_i^{-1}(x)=\left(\begin{array}{cccc}          1-x^2 & 0 &g_{i-1}^{-1}(2x- x^3) & 0 \\
                                          0 & 1& 0 & 0 \\
                                          -g_{i-1}x & 0& 1-x^2 & 0 \\
                                          0 & 0 & 0 & 1 \\
                                        \end{array}\right)$$
if $i>0$. Both matrices belong to $GL_{m_i+m_0}(B^{\Delta^k}[x])$ and are inverse to each other.
By construction, $R_i(0)=I$ and $R_i(1)=R_i$.

For each $0\leq i\leq n$ denote by $\Psi_i(x)=(m_i+m_0,\beta_i(x):A\to M_{m_i+m_0}(B^{\Delta^k}[x]))\in\cc A(A,B^{\Delta^k}[x])$, 
where $\beta_i(x)$ is the composite homomorphism
   $$A\xrightarrow{u(P_0)}\End_B(P_0)\xrightarrow{R_i(x)\tilde\sigma'R^{-1}_i(x)}M_{m_i+m_0}(B^{\Delta^k}[x]).$$
It uniquely determines an isomorphism $h_i(x):\Psi_0(x)\to\Psi_i(x)$ in $\cc A(A,B^{\Delta^k}[x])$.
Namely, there is a canonical isomorphism of $(A,B)$-bimodules
   $$h_i(x):\beta_0(x)(1)(M_{2m_0}(B^{\Delta^k}[x]))\lra\cong{}_AP_{B^{\Delta^k}}\otimes_{B^{\Delta^k}} B^{\Delta^k}[x]\lra\cong \beta_i(x)(1)(M_{m_i+m_0}(B^{\Delta^k}[x])).$$

The assignment
      \begin{multline*}
       (g_0:\Phi_0\lra{\cong}\Phi_1,\ldots,g_{n-1}:\Phi_0\lra{\cong}\Phi_n)\in i_n\cc A(A,B^{\Delta^k})\mapsto\\
           (h_0(x):\Psi_0(x)\lra{\cong}\Psi_1(x),\ldots,h_{n-1}(x):\Psi_0(x)\lra{\cong}\Psi_n(x))\in i_n\cc A(A,B^{\Delta^k}[x])         
      \end{multline*}
is obviously consistent with simplicial structure maps in the $\Delta^\bullet$-direction. Thus we get a simplicial homotopy
   $$U:i_n\cc A(A,B^\Delta)\times\Delta[1]\to i_n\cc A^\sigma(A,B^\Delta)=\colim(i_n\cc A(A,B^\Delta)\lra\sigma i_n\cc A(A,B^\Delta)\lra\sigma\cdots)$$
such that $U_1$ equals the canonical inclusion
$\iota:i_n\cc A(A,B^\Delta)\to i_n\cc A^\sigma(A,B^\Delta)$.
The map $U_0$ takes 
$(g_0:\Phi_0\lra{\cong}\Phi_1,\ldots,g_{n-1}:\Phi_0\lra{\cong}\Phi_n)\in i_n\cc A(A,B^{\Delta^k})$ to
   $$(s_0:\sigma^{m_0}(\Phi_0)\xrightarrow{\cong}X_{m_0,m_1}\sigma^{m_1}(\Phi_0)X_{m_0,m_1}^{-1},\ldots,
       s_{n-1}:\sigma^{m_n}(\Phi_0)\xrightarrow{\cong}X_{m_0,m_n}\sigma^{m_n}(\Phi_0)X_{m_0,m_n}^{-1})\in i_n\cc A(A,B^{\Delta^k}),$$
where 
   $$X_{m_0,m_i}=\left(\begin{array}{cc}    0 & -I_{m_i} \\
                                                                    I_{m_0} & 0
                                        \end{array}\right),\quad
       X_{m_0,m_i}^{-1}=\left(\begin{array}{cc}    0 & I_{m_0} \\
                                                                    -I_{m_i} & 0
                                        \end{array}\right)$$

Each $X_{m_0,m_i}\in SL_{m_i+m_0}(\bb Z)$. The proof of Lemma~\ref{barbashev} shows that there is $L_{i}(x)\in SL_{m_i+m_0}(k[x])$
such that $L_i(0)=I$ and $L_i(1)=X_{m_0,m_i}$. We set $L_0(x)=I_{2m_0}$ and $m:=m_0+\cdots+m_n$. 
For each $0\leq i\leq n$ denote by $\Theta_i(x)=(m,\gamma_i(x):A\to M_{m}(B^{\Delta^k}[x]))\in\cc A(A,B^{\Delta^k}[x])$, 
where $\gamma_i(x)$ is the composite homomorphism
   $$A\xrightarrow{u(P_0)}\End_B(P_0)\hookrightarrow M_{m_0}(B^{\Delta^k})
       \xrightarrow{L_i(x)\bar\sigma L^{-1}_i(x)}M_{m_0+m_i}(B^{\Delta^k}[x])\lra{\bar\sigma }M_{m}(B^{\Delta^k}[x]),$$
where $\bar\sigma$ maps $M_{m_0}(B^{\Delta^k})$ to the upper left corner of $M_{m_0+m_i}(B^{\Delta^k}[x])$.
As above, it uniquely determines an isomorphism $q_i(x):\Theta_0(x)\to\Theta_i(x)$ in $\cc A(A,B^{\Delta^k}[x])$.

The assignment
      \begin{multline*}
       (g_0:\Phi_0\lra{\cong}\Phi_1,\ldots,g_{n-1}:\Phi_0\lra{\cong}\Phi_n)\in i_n\cc A(A,B^{\Delta^k})
       \mapsto\\
           (q_0(x):\Theta_0(x)\lra{\cong}\Theta_1(x),\ldots,q_{n-1}(x):\Theta_0(x)\lra{\cong}\Theta_n(x))\in i_n\cc A(A,B^{\Delta^k}[x])         
      \end{multline*}
is obviously consistent with simplicial structure maps in the $\Delta^\bullet$-direction. Thus we get a simplicial homotopy
   $$V:i_n\cc A(A,B^\Delta)\times\Delta[1]\to i_n\cc A^\sigma(A,B^\Delta)$$
such that $V_0$ equals $\eta_{0,n}p_n\iota$.
The map $V_1$ takes $(g_0:\Phi_0\lra{\cong}\Phi_1,\ldots,g_{n-1}:\Phi_0\lra{\cong}\Phi_n)$ to
   $$(s_0:\sigma^{m_0}(\Phi_0)\xrightarrow{\cong}X_{m_0,m_1}\sigma^{m_1}(\Phi_0)X_{m_0,m_1}^{-1},\ldots,
       s_{n-1}:\sigma^{m_n}(\Phi_0)\xrightarrow{\cong}X_{m_0,m_n}\sigma^{m_n}(\Phi_0)X_{m_0,m_n}^{-1}).$$

The homotopies $U,V$ together yield a homotopy
   $$H:i_n\cc A(A,B^\Delta)\times\sd^1\Delta[1]\to i_n\cc A^\sigma(A,B^\Delta)$$
such that $H_1$ is the canonical inclusion
$\iota:i_n\cc A(A,B^\Delta)\to i_n\cc A^\sigma(A,B^\Delta)$ 
and $H_0$ equals $\eta_{0,n}p_n\iota$. By construction, the homotopy $H$ preserves
the trivial string of identities of the zeroth $(A,B)$-bimodule 0. Moreover, $H$ is functorial in structure maps of the 
cosimplicial $k$-algebra $A\odot S^1$ (also see the proof of Proposition~\ref{segalcat}, where the description of the category
$\cc A(A\times A,B)$ is given). Thus
there is a homotopy, denoted by the same letter, between connected simplicial sets 
   $$H:i_n\cc A(A\odot S^1,B^\Delta)\times\sd^1\Delta[1]\to i_n\cc A^\sigma(A\odot S^1,B^\Delta)$$
such that $\pi_\ell(H_0)=\pi_\ell(\eta_{0,n})\pi_\ell(p_n)\pi_\ell(\iota)$ and $\pi_\ell(H_1)=\pi_\ell(\iota)$ for all $\ell\geq 0$.
We see that $\pi_\ell(\iota)=\pi_\ell(\eta_{0,n})\pi_\ell(p_n)\pi_\ell(\iota)$. On the other hand,
$\pi_\ell(p_n)\pi_\ell(\eta_{0,n})=\id$. 
This is enough to conclude that $\eta_1$ is a weak equivalence.
This completes the proof of the theorem.
\end{proof}

\section{Symmetric matrix motives}\label{sectionsymm}

\begin{definition}\label{mtrmotsym}
The {\it symmetric matrix motive $M^\Sigma_{\mathrm{mtr}}(A)$ of a $k$-algebra $A\in\aha$} is the spectral functor taking
$B\in\aha$ to the symmetric $S^1$-spectrum 
   $$M^\Sigma_{\mathrm{mtr}}(A)(B):=(\Mtr(A,B^\Delta),\Mtr(A\odot S^1,M_2(B^\Delta)),\Mtr(A\odot S^2,M_{2^2}(B^\Delta)),\ldots).$$
The left action of $\Sigma_n$ on $M^\Sigma_{\mathrm{mtr}}(A)(B)$ is defined similarly to~\eqref{sigmaseqq}.
The symmetric spectrum structure is defined similarly to $\Mtrc_*(A,B)$.
Note that the definition of the matrix motive of an algebra is similar to that for symmetric framed motives in the sense of~\cite{GTLMS}.
\end{definition}

\begin{remark}\label{mtrvar}
The symmetric matrix motive in the sense of the
preceding definition is a bit different from the matrix motive of
$A$ in the sense of Definition~\ref{mtrmot} defined as
   $${M}_{\mathrm{mtr}}(A):=(C_*\Mtr(A,-),C_*\Mtr(A\odot S^1,-),C_*\Mtr(A\odot S^2,-),\ldots).$$
There is a canonical morphism of functors of non-symmetric $S^1$-spectra
   $$\mu:{M}_{\mathrm{mtr}}(A)\to {M}_{\mathrm{mtr}}^\Sigma(A)$$
defined as follows. For any $B\in\aha$ and $n\geq 0$ the map of simplicial sets
   $$\mu_n:\Mtr(A\odot S^n,B^\Delta)\to \Mtr(A\odot S^n,M_{2^n}(B^\Delta))$$
is induced by the non-unital homomorphism $\sigma^n:B\to M_{2^n}(B)$.
\end{remark}

\begin{lemma}\label{mtrsymmvar}
For any $A\in\aha$ the map $\mu:M_{\mathrm{mtr}}(A)\to{M}^\Sigma_{\mathrm{mtr}}(A)$ is a sectionwise level weak
equivalence of simplicial sets in positive degrees.
\end{lemma}

\begin{proof}
This follows from Lemma~\ref{magnitka}.
\end{proof}

Though ${M}_{\mathrm{mtr}}(A)$ is canonically a symmetric
$S^1$-spectrum in $Sp^\Sigma_{S^1}$, where $\Sigma_n$ acts on each space by permuting
$S^n$, the point is that it is not an $\Mtrc_*(k)$-module in
contrast with ${M}^\Sigma_{\mathrm{mtr}}(A)$. We should notice that $\Ob(\Mtrc_*(k))$ is a large set. Passing to the next universe, we may 
deal with the category of left modules $\Mtrc_*(k)\Mod$ consisting of the spectral covariant functors from
$\Mtrc_*(k)$ to $Sp^\Sigma_{S^1}$.

\begin{proposition}\label{mtrmodulevar}
The following statements are true:
\begin{enumerate}
\item ${M}^\Sigma_{\mathrm{mtr}}(A)$ is an $\Mtrc_*(k)$-module;
\item the canonical map
$\alpha:C_*\Mtrc_*(A,-)\to M_{\mathrm{mtr}}^\Sigma(A)$ in $\Mtrc_*(k)\Mod$ 
is a sectionwise $\pi_*$-isomorphism (i.e. a stable equivalence of
ordinary spectra).
\end{enumerate}
\end{proposition}

\begin{proof}
(1). ${M}^\Sigma_{\mathrm{mtr}}(A)$ is an $\Mtrc_*(k)$-module for the same reasons as 
the representable module $\Mtrc_*(A,-)$ is.

(2). The proof is like that of~\cite[Proposition~4.12(3)]{GTLMS}. Let $\Theta^\infty_{S^1}$ be the naive stabilisation functor of
$S^1$-spectra. It has the property that $X\to\Theta^\infty_{S^1}(X)$
is a stable equivalence for every $S^1$-spectrum
$X$~\cite[Proposition~4.7]{H}.

Let $\cc A_{n,r}:=\Hom_{\aha}(A\odot S^r,M_{2^n}(B^\Delta))$,
$\cc B_{m,n,r}:=\Hom_{\aha}(A\odot S^r,M_{2^{m+n}}(B^\Delta))$, $m,n,r\geq
0$. We have maps of spaces
   $$\cc A_{n,r}\xrightarrow{\sigma_B}\cc A_{n+1,r}\quad\textrm{and}\quad
     \cc B_{m,n,r}\xrightarrow{\sigma_A}\cc B_{1+m,n,r}\xrightarrow{\sigma_B}\cc B_{1+m,n+1,r}.$$
Let $\cc A_r:=\colim_{n\geq r}\cc A_{n,r}$, $\cc B_{m,r}:=\colim_{n}\cc B_{m,n,r}$,
$\cc B_r:=\colim_{m,n}\cc B_{m,n,r}$. The spaces $\cc A_r$, $\cc B_r$ constitute
right $S^1$-spectra $\cc A$ and $\cc B$. 
Each structure map
$\cc A_r\to\uhom(S^1,\cc A_{r+1})$ is the composite map determined by the following commutative diagram
   $$\xymatrix{\cc A_r:&(A\odot S^r,M_{2^r}(B^\Delta))\ar[d]_{M_{2^{r}}(\sigma_B)}\ar[r]^{M_{2^{r}}(\sigma_B)}
                             &(A\odot S^r,M_{2^{r+1}}(B^\Delta))\ar[d]_{M_{2^{r+1}}(\sigma_B)}\ar[r]^(.7){M_{2^{r+1}}(\sigma_B)}&\cdots\\
                             &(A\odot S^r,M_{2^{r+1}}(B^\Delta))\ar[d]_{-\wedge S^1}\ar[r]^{M_{2^{r+1}}(\sigma_B)}
                             &(A\odot S^r,M_{2^{r+2}}(B^\Delta))\ar[d]_{-\wedge S^1}\ar[r]^(.7){M_{2^{r+2}}(\sigma_B)}&\cdots\\
   \uhom(S^1,\cc A_{r+1}):&(S^1,(A\odot S^{r+1},M_{2^{r+1}}(B^\Delta)))\ar[r]^{(S^1,\sigma_B)}
   &(S^1,(A\odot S^{r+1},M_{2^{r+2}}(B^\Delta)))\ar[r]^(.8){(S^1,\sigma_B)}&\cdots}$$
Each structure map $\cc B_r\to\uhom(S^1,\cc B_{r+1})$ is defined 
in a similar fashion.

There is a commutative diagram of
spectra
   $$\xymatrix{\Mtrc_*(A,B^\Delta)\ar@/^1pc/[rr]^{j_{\Mtrc_*}}\ar[d]_\alpha\ar[r]_(.7)a
                &\cc A\ar[d]_c\ar[r]_{j_{\cc A}}&\Theta^\infty_{S^1}(\cc A)\ar[d]^{\Theta^\infty_{S^1}(c)}\\
                M_{\mathrm{mtr}}^\Sigma(A)(B)\ar@/_1pc/[rr]_{j_{M_{\mathrm{mtr}}^\Sigma}}\ar[r]^(.7)b&\cc B\ar[r]^{j_{\cc B}}&\Theta^\infty_{S^1}(\cc B)}$$
in which
$\Theta^\infty_{S^1}(\cc A)=\Theta^\infty_{S^1}(\Mtrc_*(A,B^\Delta))$,
$\Theta^\infty_{S^1}(\cc B)=\Theta^\infty_{S^1}(M_{\mathrm{mtr}}^\Sigma(A)(B))$.
The maps $a,b,c$ are defined in a canonical way. By the
two-out-of-three property $a,b$ are stable equivalences. Therefore
$\alpha$ is a stable equivalence if and only if $c$ is.

But $c$ is the infinite composition
   $$\cc A_r\cong\cc B_{0,r}\to\cc B_{1,r}\to\cc B_{2,r}\to\cdots\to\cc B_r.$$
Our proof will follow if we show that each map $\cc B_{m,r}\to\cc B_{1+m,r}$ is a weak
equivalence for any $r>0$. The proof of this claim is similar to that of~\cite[Lemma~4.12]{GN}.
The map in question is obtained as the colimit of the maps
   \begin{equation}\label{rrr}
    \sigma_A:(A\odot S^r,M_{2^{m+n}}(B^\Delta))\to(A\odot S^r,M_{2^{1+m+n}}(B^\Delta)).
   \end{equation}
Consider the triangle
\[
\xymatrix{
(A\odot S^r,M_{2^{1+m+n}}(B^\Delta))& (A\odot S^r,M_{2^{1+m}}(M_{2^{n-1}}(B^\Delta))) \ar[l]_(.52){(\sigma_B)_*}\\
(A\odot S^r,M_{2^{m+n}}(B^\Delta))\ar[u]^{\sigma_A}\ar[ru]_{\cong}}
\]
where the vertical map is the map~\eqref{rrr}, the skew map is the
isomorphism given by identification $M_{2^{m+n}}(B^\Delta)=M_{2^{1+m}}(M_{2^{n-1}}(B^\Delta))$, 
and the horizontal map is induced by the stabilization map in the $n$th direction.
The composite map of the triangle differs from the left vertical map by the shuffle
permutation action $\chi_{m+n,1}$ on $M_{2^{1+m+n}}(B^\Delta)$ . 
The proof of Lemma~\ref{pin} shows that the triangle is commutative
up to a simplicial homotopy. Note that the horizontal map induces an
isomorphism on the colimit over $n$. Thus the vertical map induces a
bijection on the colimits of homotopy groups $\pi_*$ as all spaces in question are connected due to
our assumption that $r>0$. We see
that the vertical map induces a weak equivalence.
\end{proof}

The Comparison Theorem~\ref{compar}, Lemma~\ref{mtrsymmvar} and Proposition~\ref{mtrmodulevar}
imply the following result.

\begin{theorem}\label{wunderbar}
There is a zigzag of stable equivalences of non-symmetric spectra, functorial in $A,B\in\aha$,
   $$\Mtrc_*(A,B^\Delta)\xrightarrow{\sim}M_{\mathrm{mtr}}^\Sigma(A)(B)\xleftarrow{\sim}M_{\mathrm{mtr}}(A)(B)
       \xrightarrow{\sim}\bb K^\sigma(A,B^\Delta)\xleftarrow{\sim}\bb K(A,B^\Delta).$$
\end{theorem}

\section{The triangulated category of symmetric matrix motives}\label{sectiontriang}

As we work with (unital) $k$-algebras only in the previous sections, we will restrict
ourselves to smooth affine algebraic varieties of finite type $\smaff_k$ over a field $k$ whenever we deal with
applications to motivic homotopy theory. This does not affect the 
general machinery for motivic spaces (or spectra) as we can safely 
take the following model for the category of pointed motivic spaces $\cc M$:
it consists of simplicial pointed presheaves on smooth affine algebraic varieties $\smaff_k$
endowed with the (projective or injective) local/motivic model structure (see, e.g.,~\cite{AHW1}). Note
that $\smaff_k$ is a skelletaly small category.

The triangulated category of $K$-motives $DK_{-}^{eff}(k)$ in the sense of~\cite{GP1} is constructed
out of the category of right $\bb K$-modules $\Mod\bb K$ associated to the spectral category
$\bb K$ of Theorem~\ref{semga}. By definition, 
   $$\bb K(X,Y):=\bb K(k[Y],k[X]),\quad X,Y\in\smaff_k.$$
A disadvantage of $\Mod\bb K$ is that it is probably not closed symmetric monoidal as the spectral
category $\bb K$ is not symmetric monoidal.

In this section an application of symmetric matrix motives is given. Namely, we suggest
yet another (genuinely local, i.e. avoiding $\bb A^1$-localization) approach to the theory of $K$-motives in the sense of~\cite{GP,GP1}.
The associated category of spectral modules will be closed symmetric monoidal. Its closed symmetric
monoidal homotopy category is compactly generated triangulated with compact generators given by 
symmetric matrix motives of algebraic varieties. By Theorem~\ref{wunderbar} they have the same stable motivic homotopy 
types as $K$-motives in the sense of~\cite{GP,GP1}.
For this, we have to recall some constructions and definitions of motivic homotopy theory
for spectral modules.

Let $\cc O$ be a spectral category and let $\Mod\cc O=[\cc O^{\op},Sp^\Sigma_{S^1}]$ be the
category of right $\cc O$-modules. Recall that the projective stable model
structure on $\Mod\cc O$ is defined as follows (see~\cite{SS1}). The
weak equivalences are the objectwise stable weak equivalences and
fibrations are the objectwise stable projective fibrations. The
stable projective cofibrations are defined by the left lifting
property with respect to all stable projective acyclic fibrations.

Let $\cc Q$ denote the set of elementary distinguished squares in
$\smaff_k$ (see~\cite[3.1.3]{MV})
   \begin{equation}\label{squareQ}
    \xymatrix{\ar@{}[dr] |{\textrm{$Q$}}U'\ar[r]\ar[d]&X'\ar[d]^\phi\\
              U\ar[r]_\psi&X}
   \end{equation}
and let $\cc O$ be a spectral category over $\smaff_k$. By $\cc Q_{\cc O}$ denote the
set of squares
   \begin{equation}\label{squareOQ}
    \xymatrix{\ar@{}[dr] |{\textrm{$\cc O Q$}}\cc O(-,U')\ar[r]\ar[d]&\cc O(-,X')\ar[d]^\phi\\
              \cc O(-,U)\ar[r]_\psi&\cc O(-,X)}
   \end{equation}
which are obtained from the squares in $\cc Q$ by taking $X\in
\text{AffSm}_k$ to $\cc O(-,X)$. The arrow $\cc O(-,U')\to\cc O(-,X')$
can be factored as a cofibration $\cc O(-,U')\rightarrowtail Cyl$
followed by a simplicial homotopy equivalence $Cyl\to\cc O(-,X')$.
There is a canonical morphism 
   $$A_{\cc O Q}:=\cc O(-,U)\bigsqcup_{\cc O(-,U')} Cyl\to\cc O(-,X).$$
   
\begin{definition}[see~\cite{GP,GP1}]\label{1214}
We say that $\cc O$ is {\it Nisnevich excisive\/} if for every
elementary distinguished square $Q$
the square $\cc O Q$~\eqref{squareOQ} is homotopy pushout in the
Nisnevich local model structure on $Sp_{S^1}^\Sigma(k):=Sp^\Sigma(\cc M,{S^1})$.

The {\it Nisnevich local model structure\/} on $\Mod\cc O$ is the
Bousfield localization of the stable projective model structure with
respect to the family of projective cofibrations
   \begin{equation*}\label{no}
    \cc N_{\cc O}=\{\cyl(A_{\cc O Q}\to\cc O(-,X))\}_{\cc Q_{\cc O}}.
   \end{equation*}
The homotopy category for the Nisnevich local model structure will
be denoted by $\shnis\cc O$. 
\end{definition}

Suppose $\cc O$ is symmetric monoidal. By a theorem of
Day~\cite{Day} $\Mod\cc O$ is a closed symmetric monoidal category
with smash product $\wedge$ and $\cc O(-,\pt)$ being the monoidal
unit. The smash product is defined as
   \begin{equation}\label{smash}
    M\wedge_{\cc O} N=\int^{\Ob\cc O\otimes\cc O}M(X)\wedge N(Y)\wedge\cc O(-,X\times Y).
   \end{equation}
The internal Hom functor, right adjoint to $-\wedge_{\cc O}M$, is given by
   $$\underline{\Mod}\cc O(M,N)(X):=Sp^\Sigma(M,N(X\times-))=\int_{Y\in\Ob\cc O}\underline{Sp}^\Sigma(M(Y),N(X\times Y)).$$
There is a natural isomorphism
   $$\cc O(-,X)\wedge_{\cc O}\cc O(-,Y)\cong\cc O(-,X\times Y).$$

\begin{theorem}[\cite{GP}]\label{modelmot}
Suppose $\cc O$ is a Nisnevich excisive
spectral category. Then the Nisnevich local model
structure on $\Mod\cc O$ is cellular, proper, spectral and weakly
finitely generated. Moreover, a map of $\cc O$-modules is a weak
equivalence in the Nisnevich local model
structure if and only if it is a weak equivalence in the Nisnevich local
model structure on $Sp_{S^1}^\Sigma(k)$. If
$\cc O$ is a symmetric monoidal spectral category then the
model structure on $\Mod\cc O$ is symmetric monoidal with respect
to the smash product~\eqref{smash} of $\cc O$-modules.
\end{theorem}

We can consider a spectral category $\Mtrc^{\Delta}_*(k)$ which is obtained
from $\Mtrc_*(k)$ by applying the Suslin complex to symmetric spectra of morphisms:
   $$\Mtrc^{\Delta}_*(k)(X,Y):=(\Mtrc_0(k[Y],k[X]^\Delta),\Mtrc_1(k[Y],k[X]^\Delta),\ldots).$$
It is symmetric monoidal by Theorem~\ref{avangard}. Denote by $\Mod\Mtrc^{\Delta}_*(k)$
the closed symmetric monoidal category of right 
$\Mtrc^{\Delta}_*(k)$-modules. By definition, $\Mod\Mtrc^{\Delta}_*(k)$ is the category of contravariant 
spectral functors from $\Mtrc^{\Delta}_*(k)$ to $Sp_{S^1}^\Sigma$.

The Nisnevich local model structure on ${\Mod}\Mtrc^{\Delta}_*(k)$ and
its homotopy category $D_{\mathsf{mtr}}^\Delta(k)$ is defined similarly to Definition~\ref{1214}. 
Namely,
   $$D_{\mathsf{mtr}}^\Delta(k):=\shnis\Mtrc^{\Delta}_*(k).$$
We call the category $D_{\mathsf{mtr}}^{\Delta}(k)$ the {\it triangulated category of symmetric matrix motives}.

In this setting, 
the {\it symmetric matrix motive $M^\Sigma_{\mathrm{mtr}}(X)$ of $X\in\smaff_k$} is the spectral functor taking
$Y\in\smaff_k$ to the symmetric $S^1$-spectrum $M^\Sigma_{\mathrm{mtr}}(k[X])(k[Y])$, where $M^\Sigma_{\mathrm{mtr}}(k[X])$
is the symmetric matrix motive of the $k$-algebra $k[X]$ in the sense of Definition~\ref{mtrmotsym}.

The proof of the following statement is the same with that for Proposition~\ref{mtrmodulevar}.

\begin{proposition}\label{frmodulevar}
Given a field $k$ and $X\in\smaff_k$, the following
statements are true:
\begin{enumerate}
\item $M^\Sigma_{\mathrm{mtr}}(X)$ is an $\Mtrc^{\Delta}_*(k)$-module;

\item the canonical map
$\alpha:\Mtrc^{\Delta}_*(-,X)\to M^\Sigma_{\mathrm{mtr}}(X)$ in ${\Mod}\Mtrc^{\Delta}_*(k)$ 
is a sectionwise $\pi_*$-iso\-mor\-phism (i.e. a stable equivalence of
ordinary spectra).
\end{enumerate}
\end{proposition}

\begin{theorem}\label{spectralmore}
Let $k$ be a perfect field. Then the following
statements are true:
\begin{enumerate}
\item the symmetric monoidal spectral category $\Mtrc^{\Delta}_*(k)$ is Nisnevich excisive
and the Nisnevich local model structure on $\Mod\Mtrc^{\Delta}_*(k)$ has all the properties of Theorem~\ref{modelmot};

\item the category $D_{\mathsf{mtr}}^\Delta(k)$ is closed symmetric monoidal 
compactly generated triangulated with compact
generators being the symmetric matrix motives $\{M^\Sigma_{\mathrm{mtr}}(X)\mid X\in\smaff_k\}$;

\item the monoidal product $M^\Sigma_{\mathrm{mtr}}(X)\wedge^L M^\Sigma_{\mathrm{mtr}}(Y)$ 
in $D_{\mathsf{mtr}}^\Delta(k)$ is isomorphic to $M^\Sigma_{\mathrm{mtr}}(X\times Y)$;

\item for every $X\in\smaff_k$ the symmetric matrix motive $M^\Sigma_{\mathrm{mtr}}(X)$ is sectionwise
zig-zag stably equivalent to the $K$-motive $M_{\bb K}(X)$ in the sense of~\cite{GP1}. In particular,
   $$K_n(X)=\Hom_{D_{\mathsf{mtr}}^\Delta(k)}(M^\Sigma_{\mathrm{mtr}}(X)[n],M^\Sigma_{\mathrm{mtr}}(pt)),\quad n\in\bb Z,$$
where $pt=\spec(k)$ and $K(X)$ is Quillen's $K$-theory of $X$.
\end{enumerate}
\end{theorem}

\begin{proof}
(1). It follows from Theorem~\ref{wunderbar} that $\Mtrc^{\Delta}_*(k)$ is Nisnevich excisive
if and only if for any distinguished square~\eqref{squareQ} the square of motivic $S^1$-spectra
   \begin{equation}\label{squareOQK}
    \xymatrix{\bb K(\Delta^\bullet\times-,U')\ar[r]\ar[d]&\bb K(\Delta^\bullet\times-,X')\ar[d]\\
              \bb K(\Delta^\bullet\times-,U)\ar[r]&\bb K(\Delta^\bullet\times-,X)}
   \end{equation}
is locally homotopy pushout.

For any $M\in\Mod\bb K$ the presheaves $\pi_i(C_*(M))$, $i\in\bb Z$, are homotopy invariant
and have $\bb K_0$-transfers, where $\bb K_0$ is the ringoid $\pi_0(\bb K)$. Since $\bb K$ is a strict
$V$-spectral category by~\cite[Theorem~4.23]{GP1}, each Nisnevich sheaf
$\pi_i^{nis}(C_*(M))$ is strictly homotopy invariant and has
$\bb K_0$-transfers. By~\cite[Theorem~6.2.7]{Mor} $C_*(M)$ is $\bb
A^1$-local in the motivic model category structure on
$Sp_{S^1}^\Sigma(k)$. It follows that 
the local fibrant replacement $C_*(M)_f$ of $C_*(M)$ is
motivically fibrant.

As $\bb K$ is Nisnevich excisive by~\cite[Theorem~4.23]{GP1}, it follows that the square
of motivically fibrant presheaves of motivic spectra
   \begin{equation*}\label{squareOQKF}
    \xymatrix{\bb K(\Delta^\bullet\times-,U')_f\ar[r]\ar[d]&\bb K(\Delta^\bullet\times-,X')_f\ar[d]\\
              \bb K(\Delta^\bullet\times-,U)_f\ar[r]&\bb K(\Delta^\bullet\times-,X)_f}
   \end{equation*}
is homotopy pushout in the motivic model category structure on $Sp_{S^1}^\Sigma(k)$. 
Thus the latter square is sectionwise homotopy pushout, and so the square~\eqref{squareOQK}
is locally homotopy pushout.

The fact that the Nisnevich local model structure on 
$\Mod\Mtrc^{\Delta}_*(k)$ has all the properties of Theorem~\ref{modelmot} now follows from the 
fact that $\Mtrc^{\Delta}_*(k)$ is Nisnevich excisive symmetric monoidal.

(2). $D_{\mathsf{mtr}}^\Delta(k)$ is closed symmetric monoidal 
compactly generated triangulated with compact
generators being the representable $\Mtrc^{\Delta}_*(k)$-modules $\{\Mtrc^{\Delta}_*(-,X)\mid X\in\smaff_k\}$.
It remains to apply Proposition~\ref{frmodulevar}(2) to show that
compact generators can be given by the symmetric matrix motives $\{M^\Sigma_{\mathrm{mtr}}(X)\mid X\in\smaff_k\}$.

(3). The isomorphism $M^\Sigma_{\mathrm{mtr}}(X)\wedge^L M^\Sigma_{\mathrm{mtr}}(Y)\cong M^\Sigma_{\mathrm{mtr}}(X\times Y)$ in
$D_{\mathsf{mtr}}^\Delta(k)$ follows from the isomorphism 
   $$\Mtrc^{\Delta}_*(-,X\times Y)\cong\Mtrc^{\Delta}_*(-,X)\wedge\Mtrc^{\Delta}_*(-,Y)
       \cong\Mtrc^{\Delta}_*(-,X)\wedge^L\Mtrc^{\Delta}_*(-,Y)$$
and Proposition~\ref{frmodulevar}(2).

(4). The fact that $M^\Sigma_{\mathrm{mtr}}(X)$ is sectionwise
zig-zag stably equivalent to the $K$-motive $M_{\bb K}(X)$ follows from Theorem~\ref{wunderbar}.
Since $M_{\bb K}(pt)$ represents Quillen's $K$-theory motivic $S^1$-spectrum by~\cite[Theorem~5.11]{GP1} and
   $$K_n(X)=\Hom_{\shnis(k)}(\Sigma^\infty_{S^1}X_+[n],M^\Sigma_{\mathrm{mtr}}(pt))
       \cong\Hom_{D_{\mathsf{mtr}}^\Delta(k)}(M^\Sigma_{\mathrm{mtr}}(X)[n],M^\Sigma_{\mathrm{mtr}}(pt)),\quad n\in\bb Z,$$
our theorem follows.
\end{proof}

\section{The triangulated category of big symmetric matrix motives}\label{sectiontriangbig}

Let $\Fr_0(k)$ be the category whose objects are those of $\smaff_k$ and
whose morphism set between $X$ and $Y$ is given by
the set of framed correspondences of level zero \cite[Example 2.1]{Voe2}, \cite[Definition 2.1]{GP3}.
As we mentioned in Section~\ref{sectionmtr} $\Fr_0(k)$ is dual to the category $\Mtr_0(k)$ 
of smooth commutative $k$-algebras and
non-unital homomorphisms.
As it is shown in~\cite[Section~5]{GP3}, the category of framed correspondences
of level zero $\Fr_0(k)$ has an action by finite pointed sets $Y\otimes
K:=\bigsqcup_{K\setminus *}Y$ with $Y\in\smaff_k$ and $K$ a finite pointed
set. The cone of $Y$ is the simplicial object $Y\otimes I$ in
$\Fr_0(k)$, where $(I,1)$ is the pointed simplicial set $\Delta[1]$
with basepoint 1. There is a natural morphism $i_0:Y\to Y\otimes I$
in $\Delta^{\textrm{\op}}\Fr_0(k)$. Given a closed inclusion of smooth affine
schemes $j: Y\hookrightarrow X$, denote by $X//Y$ a simplicial
object in $\Fr_0(k)$ which is obtained by taking the pushout of the diagram
   $X\hookleftarrow Y\bl{i_0}\hookrightarrow Y\otimes I$
in $\Delta^{\textrm{\op}}\Fr_0(k)$.
The simplicial object $X//Y$ termwise equals
   $X,X\sqcup Y,X\sqcup Y\sqcup Y,\ldots$ .
By
$\bb G_m^{\wedge 1}$ we mean the simplicial object $\bb G_m//\{1\}$ in $\Fr_0(k)$.
It looks termwise as
   $$\bb G_m,\bb G_m\sqcup pt,\bb G_m\sqcup pt\sqcup pt,\ldots$$

Given $\cc X\in\Mod\Mtrc^{\Delta}_*(k)$, we set 
   $$\cc X\langle\bb G_m^{\wedge n}\rangle:=\cc X\wedge_{\Mtrc^{\Delta}_*(k)}\Mtrc^{\Delta}_*(-,\bb G_m^{\wedge n}),\quad n\geq 0,$$
where $\Mtrc^{\Delta}_*(k)(-,\bb G_m^{\wedge n})=\Mtrc^{\Delta}_*(k)(-,\bb G_m^{\wedge 1})
\wedge_{\Mtrc^{\Delta}_*(k)}\bl n\cdots\wedge_{\Mtrc^{\Delta}_*(k)}\Mtrc^{\Delta}_*(k)(-,\bb G_m^{\wedge 1})$.
There is a canonical map induced by the adjunction unit morphism
   \begin{equation}\label{wedgegm}
    \cc X\langle\bb G_m^{\wedge n}\rangle\to
    \uhom_{\Mod\Mtrc^{\Delta}_*(k)}(\Mtrc^{\Delta}_*(-,\bb G_m^{\wedge 1}),\cc X\langle\bb G_m^{\wedge n+1}\rangle).
   \end{equation}

We can stabilize our constructions in the $\bb G_m^{\wedge 1}$-direction as follows. Denote by 
$-\boxtimes\bb G_m^{\wedge 1}$ the endofunctor $\cc X\in\Mod\Mtrc^{\Delta}_*(k)\mapsto\cc X\langle\bb G_m^{\wedge 1}\rangle$.
Following Hovey~\cite[Section~8]{H}, we consider the stable model structure on $\bb G_m^{\wedge 1}$-symmetric spectra 
$Sp^\Sigma(\Mod\Mtrc^{\Delta}_*(k),\bb G_m^{\wedge 1})$ (we start with the Nisnevich local stable model structure
on $\Mod\Mtrc^{\Delta}_*(k)$). Its homotopy category is denoted by $D_{\mathsf{mtr},\bb G_m}^{\Delta}(k)$ and called
the {\it triangulated category of big symmetric matrix motives}.
It is worth mentioning that the definition of $D_{\mathsf{mtr},\bb G_m}^{\Delta}(k)$ is genuinely local, i.e. it avoids $\bb A^1$-localization.

Note that the $\ell$th space of the symmetric matrix motive $M^\Sigma_{\mathrm{mtr}}(X)$ of $X\in\smaff_k$ is a colimit 
   $$M^\Sigma_{\mathrm{mtr}}(X)_\ell=\colim(\Mtrc^{\Delta}_0(k)(M_{2^\ell}(-),X)\to\Mtrc^{\Delta}_{1}(M_{2^\ell}(-),X)\to\Mtrc^{\Delta}_{2}(M_{2^\ell}(-),X)\to\cdots),$$
where the $n$th map takes a section $f:k[X]\odot S^n\to M_{2^{n}}(M_{2^\ell}(k[Y]^\Delta))$ on $Y\in\smaff_k$ to the composition
   $$k[X]\odot S^n\xrightarrow{f}M_{2^n}(M_{2^\ell}(k[Y]^\Delta))\xrightarrow{M_{2^n}(\sigma_{M_{2^\ell}(k[Y]}))}
       M_{2^{n+1}}(M_{2^\ell}(k[Y]^\Delta))\xrightarrow{\chi_{n,1}}M_{2^{1+n}}(M_{2^\ell}(k[Y]^\Delta)).$$

The map~\eqref{wedgegm} yields a morphism
   $$\Mtrc^{\Delta}_*(-,X\times\bb G_m^{\wedge n})\to
      \uhom_{\Mod\Mtrc^{\Delta}_*(k)}(\Mtrc^{\Delta}_*(-,\bb G_m^{\wedge 1}),\Mtrc^{\Delta}_*(-,X\times\bb G_m^{\wedge n+1})).$$
It fits in a commutative diagram
   $$\xymatrix{\Mtrc^{\Delta}_*(-,X\times\bb G_m^{\wedge n})\ar[r]\ar[d]&
      \uhom_{\Mod\Mtrc^{\Delta}_*(k)}(\Mtrc^{\Delta}_*(-,X\times\bb G_m^{\wedge 1}),\Mtrc^{\Delta}_*(-,\bb G_m^{\wedge n+1}))\ar[d]\\
      M^\Sigma_{\mathrm{mtr}}(X\times\bb G_m^{\wedge n})\ar[r]^(.3)\mu&
      \uhom_{\Mod\Mtrc^{\Delta}_*(k)}(\Mtrc^{\Delta}_*(-,\bb G_m^{\wedge 1}),M^\Sigma_{\mathrm{mtr}}(X\times\bb G_m^{\wedge n+1}))}$$

Due to Theorem~\ref{spectralmore} $\Mtrc^{\Delta}_*(k)$ is Nisnevich excisive. Therefore
the local fibrant replacement $\cc X_f$ of a right $\Mtrc^{\Delta}_*(k)$-module $\cc X$ must
also be a local replacement of $\cc X$ in $Sp^\Sigma_{S^1}(k)$ after forgetting the module stucture
by Theorem~\ref{modelmot}.

\begin{theorem}[Cancellation Theorem for matrix motives]\label{cancel}
Let $X\in\smaff_k$ and $n\geq 0$.
Then the following statements are true:

$(1)$ the canonical morphism 
   $$\mu:M^\Sigma_{\mathrm{mtr}}(X\times\bb G_m^{\wedge n})\to
       \uhom_{\Mod\Mtrc^{\Delta}_*(k)}(\Mtrc^{\Delta}_*(-,\bb G_m^{\wedge 1}),M^\Sigma_{\mathrm{mtr}}(X\times\bb G_m^{\wedge n+1}))$$
in $\Mod\Mtrc^{\Delta}_*(k)$ is a sectionwise stable equivalence of $S^1$-spectra;

$(2)$ the induced map in $\Mod\Mtrc^{\Delta}_*(k)$ 
   $$M^\Sigma_{\mathrm{mtr}}(X\times\bb G_m^{\wedge n})_f\to\uhom_{\Mod\Mtrc^{\Delta}_*(k)}(\Mtrc^{\Delta}_*(-,\bb G_m^{\wedge 1}),M^\Sigma_{\mathrm{mtr}}(X\times\bb G_m^{\wedge n+1})_f)$$
is a schemewise stable equivalence of $S^1$-spectra whenever the base field $k$ is perfect;

$(3)$ statements $(1)$ and $(2)$ are also valid for $\Mtrc^{\Delta}_*(-,X\times\bb G_m^{\wedge n})$.
\end{theorem}

\begin{proof}
(1). Due to Theorem~\ref{wunderbar} the statement reduces to Cancellation Theorem for $\bb K$-motives.
The latter is proven similarly to~\cite[Theorem~4.6]{GP2}. Namely, consider the Grayson 
tower for $\bb K$-motives~\cite[Definition 7.12]{GP}
   $$\cdots\to\bb K(\Delta^\bullet\times-,X\times\bb G_m^{\wedge 2})\to
       \bb K(\Delta^\bullet\times-,X\times\bb G_m^{\wedge 1})\to\bb K(\Delta^\bullet\times-,X).$$
Its layers are Eilenberg--Mac Lane $\bb K_0$-modules associated to simplicial Abelian presheaves (up to shift)
$\bb K_0(\Delta^\bullet\times-,X\times\bb G_m^{\wedge n})$. By a theorem of Suslin~\cite[Theorem~4.11]{S}
the canonical morphism
   $$\bb K_0(\Delta^\bullet\times-,X\times\bb G_m^{\wedge n})\to
       [\bb G_m^{\wedge 1},\bb K_0(\Delta^\bullet\times-,X\times\bb G_m^{\wedge n+1})]$$
is a sectionwise weak equivalence. Similarly to~\cite[Theorem~7.13]{GP}
the Grayson tower yields a strongly convergent spectral sequence with $E_2^{*,*}$-page being 
$\pi_*(\bb K_0(\Delta^\bullet\times-,X\times\bb G_m^{\wedge *}))$ (up to shift of indices). The Suslin 
theorem now implies the claim (see the proof of~\cite[Theorem~4.6]{GP2} as well).

(2). The proof now literally repeats that for framed motives of algebraic varieties~\cite[Theorem~A(2)]{AGP} if we use 
properties of homotopy invariant presheaves with $\bb K_0$-transfers (see~\cite[Section~4]{S}) instead of framed transfers.
An alternative proof is like that for~\cite[Theorem~4.6]{GP2} if we use~\cite[Corollary~4.14]{S}.

(3). This statement immediately follows from $(1)$, $(2)$ and Proposition~\ref{mtrmodulevar}(2).
\end{proof}

\begin{corollary}\label{bispectrum}
Let $k$ be a perfect field and $X\in\smaff_k$.
Then the bispectrum
   $$M_{\mathrm{mtr}}^{\Sigma,\bb G}(X)_f=(M^\Sigma_{\mathrm{mtr}}(X)_f,M^\Sigma_{\mathrm{mtr}}(X\times\bb G_m^{\wedge 1})_f,
       M^\Sigma_{\mathrm{mtr}}(X\times\bb G_m^{\wedge 2})_f,\ldots)$$
with structure maps those of
Theorem~\ref{cancel}(2) is a motivically fibrant $(S^1,\bb G^{\wedge 1}_m)$-bispectrum.
\end{corollary}

Given $\cc X\in D_{\mathsf{mtr}}^{\Delta}(k)$ we write $\cc X(1)$ to denote $\cc X\boxtimes^L\bb G_m^{\wedge 1}$.

\begin{corollary}\label{888}
Let $k$ be a perfect field. Then for any $\cc X,\cc Y\in D_{\mathsf{mtr}}^{\Delta}(k)$ the map
   $$\Hom_{D_{\mathsf{mtr}}^{\Delta}(k)}(\cc X,\cc Y)\to\Hom_{D_{\mathsf{mtr}}^{\Delta}(k)}(\cc X(1),\cc Y(1))$$
is an isomorphism.
\end{corollary}

Define a suspension functor 
 $\Sigma^\infty_{\bb G_m}$ from ${\Mod}\Mtrc^{\Delta}_*(k)$ to $Sp^\Sigma(\Mod\Mtrc^{\Delta}_*(k),\bb G_m^{\wedge 1})$ as
   $$\Sigma^\infty_{\bb G_m}(\cc X):=(\cc X(\pt),\cc X\langle\bb G_m^{\wedge 1}\rangle,\cc X\langle\bb G_m^{\wedge 2}\rangle,\ldots).$$
Each structure map is induced by~\eqref{wedgegm}. It is left adjoint to the functor taking an object of 
$Sp^\Sigma(\Mod\Mtrc^{\Delta}_*(k),\bb G_m^{\wedge 1})$ to its zeroth entry.

Denote by $D_{\mathsf{mtr},\bb G_m}^{\Delta,\mathsf{eff}}(k)$ the full triangulated subcategory of
$D_{\mathsf{mtr},\bb G_m}^{\Delta}(k)$ compactly generated by the suspension $\Mtrc^{\Delta}_*(k)$-modules
$\{\Sigma^\infty_{\bb G_m}(\Mtrc^{\Delta}_*(-,X))\mid X\in\smaff_k\}$.

\begin{corollary}\label{895}
Let $k$ be a perfect field. Then the suspension functor induces an equivalence of compactly generated triangulated categories
   $$L\Sigma^\infty_{\bb G_m}:D_{\mathsf{mtr}}^{\Delta}(k)\lra{\simeq}D_{\mathsf{mtr},\bb G_m}^{\Delta,\mathsf{eff}}(k).$$
\end{corollary}

\begin{proof}
By Corollary~\ref{bispectrum} there is an isomorphism in $D_{\mathsf{mtr},\bb G_m}^{\Delta}(k)$
   $$L\Sigma^\infty_{\bb G_m}(\Mtrc^{\Delta}_*(-,X))\cong M_{\mathrm{mtr}}^{\Sigma,\bb G}(X)_f,\quad X\in\smaff_k.$$
The left adjoint functor $L\Sigma^\infty_{\bb G_m}$ takes compact 
generators to compact generators with isomorphic Hom-sets.
It remains to apply~\cite[Lemma~4.8]{GJ}.
\end{proof}

Consider the bispectrum $KGL$ defined in~\cite[Section~6]{GP2}. It represents the classical $K$-theory
$\bb P^1$-spectrum $BGL$ in the category of bispectra~\cite[Theorem~A.1]{GP2}. 
Denote by $\kgl$ the $(S^1,\bb G^{\wedge 1}_m)$-bispectrum $f_0(KGL)$.

\begin{lemma}\label{fzero}
Let $k$ be a perfect field.
Then the bispectrum
   $$M_{\mathrm{mtr}}^{\Sigma,\bb G}(pt)=(M^\Sigma_{\mathrm{mtr}}(pt),M^\Sigma_{\mathrm{mtr}}(\bb G_m^{\wedge 1})_,
       M^\Sigma_{\mathrm{mtr}}(\bb G_m^{\wedge 2}),\ldots)$$
with structure maps those of
Theorem~\ref{cancel}(2) is isomorphic to $\mathsf{kgl}$ in $SH(k)$.
\end{lemma}

\begin{proof}
By~\cite[Lemma~7.9]{GP2} $\kgl$ is isomorphic to the $(S^1,\bb G^{\wedge 1}_m)$-bispectrum
   $$\bb K\bb G(pt):=(\bb K(-,pt),\bb K(-,\bb G_m^{\wedge 1}),\bb K(-,\bb G_m^{\wedge 2}),\ldots).$$
It follows from Theorem~\ref{wunderbar} that the latter bispectrum is isomorphic 
to $M_{\mathrm{mtr}}^{\Sigma,\bb G}(pt)$ in $SH(k)$.
\end{proof}

The following statement gives an explicit model for the commutative ring bispectrum $\kgl\in SH(k)$
in the category of symmetric bispectra and the category of $\kgl$-modules $\Mod_{SH(k)}\kgl$.

\begin{corollary}\label{fzerocor}
Let $k$ be a perfect field. Then the commutative ring object 
   $$\cc M_*^{\bb G}(pt):=(\Mtrc_*^\Delta(-,pt),\Mtrc_*^\Delta(-,\bb G_m^{\wedge 1})_,
       \Mtrc_*^\Delta(-,\bb G_m^{\wedge 2}),\ldots)$$
of $Sp^\Sigma(\Mod\Mtrc^\Delta_*(k),\bb G^{\wedge 1}_m)$ represents $\kgl$ in $SH(k)$ after forgetting the structure
of $\Mtrc_*^\Delta(k)$-modules in each level. In particular, $\Mod_{SH(k)}\kgl$ is equivalent to the homotopy category
$\Ho(\Mod\cc M_*^{\bb G}(pt))$ of 
$\cc M_*^{\bb G}(pt)$-modules taken in classical symmetric motivic $(S^1,\bb G^{\wedge 1}_m)$-bispectra 
$Sp^\Sigma_{S^1,\bb G^{\wedge 1}_m}(k)$ and equipped with stable (projective) motivic model structure.
\end{corollary}

\begin{proof}
This follows from Lemma~\ref{fzero} and Proposition~\ref{mtrmodulevar}(2).
\end{proof}

The following lemma computes a fibrant replacement of the bispectrum $\kgl\wedge X_+\in SH(k)$,
$X\in\smaff_k$, after inverting the characteristic. 

\begin{proposition}\label{kglx}
Let $k$ be a perfect field of exponential characteristic $e$ and $X\in\smaff_k$. Then the canonical map of bispectra
   $$\cc M_*^{\bb G}(pt)\wedge X_+\to M_{\mathrm{mtr}}^{\Sigma,\bb G}(X)_f$$
is an isomorphism in $SH(k)[1/e]$.
\end{proposition}

\begin{proof}
Consider a bispectrum
   $$\bb K\bb G(X):=(\bb K(-,X),\bb K(-,X\times\bb G_m^{\wedge 1}),\bb K(-,X\times\bb G_m^{\wedge 2}),\ldots).$$
By the proof of Lemma~\ref{fzero} it is enough to show that the canonical map of bispectra
   $$\bb K\bb G(pt)\wedge X_+\to \bb K\bb G(X)$$
is an isomorphism in $SH(k)[1/e]$. The Grayson tower of bispectra constructed in~\cite[p.~149]{GP2}
   $$\cdots\to\Sigma_{S^1}^2\Sigma^2_{\bb G_m}\bb K\bb G(X)\to\Sigma_{S^1}\Sigma_{\bb G_m}\bb K\bb G(X)\to\bb K\bb G(X)$$
fits in a commutative diagram in $SH(k)$
   $$\xymatrix{\cdots\ar[r]&\Sigma_{S^1}^2\Sigma^2_{\bb G_m}\bb K\bb G(pt)\wedge X_+\ar[r]\ar[d]
                       &\Sigma_{S^1}\Sigma_{\bb G_m}\bb K\bb G(pt)\wedge X_+\ar[r]\ar[d]&\bb K\bb G(pt)\wedge X_+\ar[d]\\
                       \cdots\ar[r]&\Sigma_{S^1}^2\Sigma^2_{\bb G_m}\bb K\bb G(X)\ar[r]
                       &\Sigma_{S^1}\Sigma_{\bb G_m}\bb K\bb G(X)\ar[r]&\bb K\bb G(X)}$$
The layers of the upper sequence are bispectra of the form (up to shift)
   $$M^{\bb G}_{\bb K_0}(pt)(n)\wedge X_+:=
        (\bb K_0(\Delta^\bullet,\bb G_m^{\wedge n})\wedge X_+,\bb K_0(\Delta^\bullet,\bb G_m^{\wedge n+1})\wedge X_+,\ldots),\quad n\geq 0.$$
The layers of the lower sequence are bispectra of the form (up to shift)
   $$M^{\bb G}_{\bb K_0}(X)(n):=(\bb K_0(\Delta^\bullet,X\times\bb G_m^{\wedge n}),\bb K_0(\Delta^\bullet,X\times\bb G_m^{\wedge n+1}),\ldots),\quad n\geq 0.$$
By the generalized R{\"o}ndigs--{\O}stv{\ae}r theorem~\cite[Theorem~5.3]{Gark2} all maps between the layers
$M^{\bb G}_{\bb K_0}(pt)(n)\wedge X_+\to M^{\bb G}_{\bb K_0}(X)(n)$ are isomorphisms in $SH(k)[1/e]$.
Similarly to~\cite[Theorem~4.7]{GP2} both towers give rise to a map of strongly 
convergent spectral sequences with isomorphic $E^2_{*,*}$-pages, and hence our statement follows.
\end{proof}

We are now in a position to prove the following result.

\begin{theorem}\label{kglmodules}
Let $k$ be a perfect field of exponential characteristic $e$.
There is a natural equivalence of compactly generated triangulated categories
   $$\Mod_{SH(k)[1/e]}\kgl\simeq D_{\mathsf{mtr},\bb G_m}^{\Delta}(k)[1/e].$$
\end{theorem}

\begin{proof}
By Corollary~\ref{fzerocor} $\Mod_{SH(k)}\kgl$ is equivalent to the homotopy category
$\Ho(\Mod\cc M_*^{\bb G}(pt))$. There is a canonical functor between model categories 
   $$\Mod\cc M_*^{\bb G}(pt)\to Sp^\Sigma(\Mod\Mtrc^{\Delta}_*(k),\bb G_m^{\wedge 1}).$$
It induces a functor between compactly generated triangulated categories
   $$\Ho(\Mod\cc M_*^{\bb G}(pt))[1/e]\to D_{\mathsf{mtr},\bb G_m}^{\Delta}(k)[1/e].$$
It follows from Proposition~\ref{kglx} that this functor takes compact generators to compact generators with isomorphic
Hom-sets between them. Now our theorem follows from~\cite[Lemma~4.8]{GJ}.
\end{proof}

We also refer the reader to~\cite[Corollary~4.3]{Bac} for a model of $\Mod_{SH(k)[1/e]}\kgl$ in terms of
``finite flat correspondences''.

\bibliographystyle{amsalpha}

\printindex

\end{document}